\documentclass[12pt]{amsart}
\usepackage[colorlinks=true,pagebackref,hyperindex,citecolor=green,linkcolor=red]{hyperref}
\usepackage{amsmath}
\usepackage{amsfonts}
\usepackage{amssymb}
\usepackage{setspace}
\usepackage{moreverb}
\usepackage[latin1]{inputenc}
\usepackage[T1]{fontenc}
\usepackage{amssymb}
\usepackage{tikz}
\usepackage{color}
\usepackage[all]{xy}
\usepackage{xfrac}
\usepackage[top=1in, bottom=1in, left=1in, right=1in]{geometry}
\usepackage{mathrsfs}

\usetikzlibrary{shapes}
\newtheorem{theoremx}{Theorem}
\newtheorem{theorem}{Theorem}[section]

\newtheorem{corollary}[theorem]{Corollary}
\newtheorem{lemma}[theorem]{Lemma}
\newtheorem{proposition}[theorem]{Proposition}

\theoremstyle{definition}
\newtheorem{definition}[theorem]{Definition}

\newtheorem{notation}[theorem]{Notation}

\newtheorem{example}[theorem]{Example}

\newtheorem{conjecture}[theorem]{Conjecture}

\newtheorem{remark}[theorem]{Remark}
\numberwithin{equation}{subsection}

\newcommand{\NN}{\mathbb{N}}

\newcommand{\ZZ}{\mathbb{Z}}

\newcommand{\RR}{\mathbb{R}}

\newcommand{\cI}{\mathcal{I}}

\newcommand{\m}{\mathfrak{m}}

\newcommand{\Vol}{\operatorname{Vol}}
\newcommand{\V}{\operatorname{V}}

\newcommand{\ceil}[1]{\lceil {#1} \rceil}

\newcommand{\e}{\text{e}}
\newcommand{\ehk}{\e_{HK}}

\begin{document}
\newcommand{\tens}{\otimes}
\newcommand{\hhtest}[1]{\tau ( #1 )}
\renewcommand{\hom}[3]{\operatorname{Hom}_{#1} ( #2, #3 )}

\allowdisplaybreaks

\title{$F$-Volumes}

\author[W. Badilla-C\'espedes]{W\'agner Badilla-C\'espedes{$^1$}}
\address{Centro de Investigaci\'on en Matem\'aticas, Guanajuato, Gto., M\'exico}
\email{wagner.badilla@cimat.mx }

\author[L. N\'u\~nez-Betancourt]{Luis N\'u\~nez-Betancourt{$^2$}}
\address{Centro de Investigaci\'on en Matem\'aticas, Guanajuato, Gto., M\'exico}
\email{luisnub@cimat.mx}

\author[S. Rodr\'iguez-Villalobos]{Sandra Rodr\'iguez-Villalobos{$^3$}}
\address{Department of Mathematics, University of Utah, 155 South 1400 East, 
Salt Lake City, UT 84112, USA}
\email{rodriguez@math.utah.edu}

\thanks{{$^1$}Partially supported by CONACYT Grants 295631 $\&$ 284598, and  N\'u\~nez-Betancourt's C\'atedra Marcos Moshinsky}
\thanks{{$^2$}Partially supported by CONACYT Grant 284598 and C\'atedras Marcos Moshinsky.}
\thanks{{$^3$}Partially supported by the Mexican Academy of Sciences (AMC) and  CONACYT Grant 284598.}


\maketitle

\begin{abstract}
In this work we  define a numerical invariant called $F$-volume. This number extends the definition of $F$-threshold of a pair of ideals $I$ and $J$, $c^J(I)$ to a sequence of ideals $J$, $I_1, \ldots, I_t$. We obtain several properties that emulate those of the $F$-threshold. In particular, the $F$-volume detects $F$-pure complete intersections. In addition, we relate this invariant to the Hilbert-Kunz multiplicity.
\end{abstract}

\tableofcontents

%%%%%%%%%%%%%%%%%%%%%%%%%%%%%%%%%%%%%%%%%%%%%%%%%%%%%%%%%%%%%%%
\section{Introduction}
%%%%%%%%%%%%%%%%%%%%%%%%%%%%%%%%%%%%%%%%%%%%%%%%%%%%%%%%%%%%%%%
Throughout this manuscript  $R$ denotes  a Noetherian ring of prime characteristic $p$.
Given two ideals $I,J\subseteq R$ such that $I\subseteq \sqrt{J}$, the $F$-threshold of $I$ with respect to $J$ \cite{MTW,HMTW,DSNBP}
is defined by
$$
c^J(I)=\lim\limits_{e\to\infty} \frac{\nu^J_I(p^e)}{p^e},
$$
where  $\nu^J_I(p^e)=\max\{ t \in \NN | I^t \not\subseteq J^{[p^e]}\}$.
If  $R$ is a regular ring, then the set of $F$-thresholds of $I$ is the same as the set of $F$-jumping numbers of the test ideals of $I$ \cite{BMS-MMJ}. 
If the ring is not regular, these two sets of invariants differ (e.g. \cite{MOY}).
However, these numbers still give information on $I,J$ and $R$. For instance, one can use $F$-thresholds to study integral closure, tight closure,  Hilbert-Samuel multiplicities \cite{HMTW},
 and $a$-invariants  \cite{DSNB,DSNBP}.

Motivated by the mixed test ideals associated to a sequence of ideals $\underline{I}=I_1,\ldots,I_t$ and their constancy regions, we define an analogue of the $F$-threshold for a sequence of ideals. In our main result, we  describe this invariant as a limit of a convergent sequence.

\begin{theoremx}[{see Theorem \ref{ThmLimitExistsSec} and Definition \ref{def1}}]\label{MainThmExistence}
Let $\underline{I}=I_1,\ldots, I_t\subseteq R$ be a sequence of ideals, and $J\subseteq R$ be an ideal such that $I_1,\ldots, I_t\subseteq \sqrt{J}$. 
Let
\begin{align*}
\V_{\underline{I}}^J(p^e)=\{(a_1,\ldots, a_t)\in\NN^t\; | \; 
	I_1^{a_1}\cdots I_t^{a_t}\not \subseteq J^{[p^e]}\}.
\end{align*}
Then, the limit
$$
\Vol^J_F( \underline{I})=\lim\limits_{e\to\infty} \frac{|\V_{\underline{I}}^J(p^e)|}{p^{et}}
$$
converges, and it is called the $F$-volume of $\underline{I}$ with respect to $J$.
\end{theoremx}
  If $I_i$ is a  principal ideal generated by $f_i$, we 
write $\V_{\underline{f}}^J(p^e)$ and $\Vol^{J}_F (\underline{f})$ instead of $\V_{\underline{I}}^J(p^e)$ and 
$\Vol^{J}_F (\underline{I})$ respectively, where $\underline{f}$ is the sequence of element $f_1,\ldots, f_t$.

In the regular case,  this limit is the sum of the volumes of the constancy regions where $\tau(I^{a_1}_1 \cdots I^{a_t}_t)\not\subseteq J$   \cite[Remark 3.10]{BMS-MMJ} (see also \cite{FelipeConstReg}).

The proof of Theorem \ref{MainThmExistence} is based on a technical extension of the ideas of the case for a single ideal  \cite{DSNBP}.  However, the case of multiple ideals is not a simple consequence of this case. We devote Section \ref{SecExistence} to this proof. We also show a few properties of the $F$-volume that extend those of the $F$-thresholds.

If $R$ is an $F$-pure ring, the $F$-volume is the $\ell$-volume of a set in $\RR^\ell$ (see Propositions \ref{PropOtherLim} and 
\ref{PropositionMeasureB}). In Section \ref{SecFpure}, we present this and other results for $F$-pure rings. In particular, we show 
that $F$-volumes detect $F$-pure complete intersections.

\begin{theoremx}[{see Theorem \ref{ThmFpureCI}}]
Suppose that $(R,\m,K)$ is a local regular ring.
Let $I \subseteq \mathfrak{m}$ be an ideal in $R$, and $\underline{f}=f_1,\ldots, f_t$ be minimal generators of $I$. Then, $\Vol^{\mathfrak{m}}_F (\underline{f})=1$ if and only if $I$ is an $F$-pure complete intersection. 
\end{theoremx}

In Section \ref{SecHKVol}, we relate the $F$-volume and the Hilbert-Kunz multiplicity. Specifically, we obtain the following result.

\begin{theoremx}[{see Theorem \ref{ThmHK}}]\label{MainHK}
Suppose that $(R,\m,K)$ is a local ring. Let $\underline{f}=f_1,\ldots , f_t$ be part of a system of parameters for $R$, $I=(\underline{f})$, and $S=R/I$. Then, \[\e_{HK}(JS; S)\geq \frac{\e_{HK}(J;R)}{\Vol_F^J(\underline{f})}\]
for any $\m$-primary ideal $J$, such that $I\subseteq J$.
\end{theoremx}

In Remark \ref{RemComp}, we relate Theorem \ref{MainHK} 
with a conjecture regarding the $F$-thresholds and the  Hilbert-Kunz multiplicity  \cite{NBS}.
In particular, we improve an estimate given in previous results \cite{NBS}.

%%%%%%%%%%%%%%%%%%%%%%%%%%%%%%%%%%%%%%%%%%%%%%%%%%%%%%%%
\section{Existence and definition}\label{SecExistence}
%%%%%%%%%%%%%%%%%%%%%%%%%%%%%%%%%%%%%%%%%%%%%%%%%%%%%%%%

In this section we prove a more general version of  Theorem \ref{MainThmExistence}. In order to do this, we start by introducing a couple of definitions.
\begin{definition}
A sequence $J_\bullet=\{J_{p^e}\}_{e \in \NN}$ of ideals in $R$ whose terms are indexed by the powers of the characteristic is called a $p$-family if $J_{p^e}^{[p]} \subseteq J_{p^{e+1}}$ for all  $e\in \NN$.
\end{definition}
An example of a $p$-family of ideals is the sequence  $J_\bullet=\{J^{[p^e]}\}_{e\in \NN}$ of Frobenius powers of an ideal $J$.

There are important $p$-families that relate to several limits in prime characteristic that measure singularities 
\cite{SmithVDB, HLMCM,YaoObs,TuckerFSig,DJpfamilies}.

\begin{definition} \label{def1}
Let $\underline{I}=I_1,\ldots, I_t\subseteq R$ be a sequence of ideals, and $J_\bullet=\{J_{p^e}\}_{e \in \NN}$ be a $p$-family of ideals in $R$ such that $I_1,\ldots, I_t\subseteq \sqrt{J_1}$. 
For each $e\in\NN$, we define
\begin{align*}
\V_{\underline{I}}^{J_\bullet}(p^e)=\{(a_1,\ldots, a_t)\in\NN^t\; | \; 
	I_1^{a_1}\cdots I_t^{a_t}\not \subseteq J_{p^e}\}.
\end{align*}
If $\underline{f}=f_1,\ldots,f_t$ is a sequence of elements of $R$ such that $f_1,\ldots, f_t\in  \sqrt{J_1}$, we use $\V_{\underline{f}}^{J_\bullet}(p^e)$ to denote $\V_{\underline{I}}^{J_\bullet}(p^e)$ where $\underline{I}=f_1R,\ldots, f_tR$. In case that the $p$-family is $J_\bullet=\{J^{[p^e]}\}_{e \in \NN}$ with $J$ an ideal in $R$, $\V_{\underline{I}}^{J_\bullet}(p^e)$ is denoted by $\V_{\underline{I}}^J(p^e)$. 
\end{definition}

\begin{remark}\label{finiteness}
Since $I_1,\ldots, I_t\subseteq \sqrt{J_1}$, for each $n\in\{1,\ldots, t\}$, there exists $\ell_n\in\NN$ such that $I_n^{\ell_n}\subseteq J_1$. Additionally, we have that $I_n^{\mu(I_n)p^e}\subseteq I_n^{[p^e]}$ and, as a consequence, $I_n^{\mu(I_n)\ell_np^e}\subseteq J_1^{[p^e]} \subseteq J_{p^e}$. Hence, if $I_1^{a_1}\cdots I_t^{a_t}\not \subseteq J_{p^e}$, then $a_n<\mu(I_n)\ell_np^e$ for all $n\in\{1,\ldots, t\}$. Thus, $|\V_{\underline{I}}^{J_\bullet}(p^e)|\leq p^{et}\prod_{n=1}^t\mu(I_n)\ell_n$ for every $e\in \NN$. Therefore, the sequence $\left\{\frac{|\V_{\underline{I}}^{J_\bullet}(p^e)|}{p^{et}}\right\}_{e\in\NN}$ is bounded.
\end{remark}

We now recall a well-known lemma to the experts (see for instance  \cite[Lemma $3.2$]{DSNBP}).

\begin{lemma}\label{DNBP1}
	Let  $\mathfrak{a}\subseteq R$ be an ideal. 
	Then, for every $r\geq(\mu(\mathfrak{a})+s-1)p^e$, we have that $\mathfrak{a}^r=\mathfrak{a}^{r-sp^e}(\mathfrak{a}^{[p^e]})^s$.
\end{lemma}

Towards proving Theorem \ref{MainThmExistence}, we need to introduce notation to describe different objects.
\begin{notation}
	Let $\underline{I}=I_1,\ldots, I_t\subseteq R$ be a sequence of ideals, and $a=(a_1,\ldots, a_t)\in\NN^t$. 
	We denote $(\ceil{a_1},\ldots,\ceil{a_t})$ by $\ceil{a}$.
	We write $\underline{I}^a$ to denote $I_1^{a_1}\cdots I_t^{a_t}$.
	%Additionally, we write $\underline{I}_n$ to denote the sequence of ideals $I_1,\ldots,I_{n-1},I_{n+1},\ldots, I_t$.
	Additionally, for each $x=(x_1,\ldots,x_t)\in \RR^t$, we write $\widehat{x}^n$ to denote $(x_1,\ldots,x_{n-1}, x_{n+1},\ldots, x_t)$.
	Let $e_1, e_2\in \NN$. Let $C$ be a subset of $\frac{1}{p^{e_1}}\NN^t$. 
	We denote the set 
	\[\bigcup_{x=(x_1,\ldots, x_t)\in C}\left\{y=(y_1,\ldots, y_t)\in \frac{1}{p^{e_1+e_2}}\NN^t: x_i-\frac{1}{p^{e_1}}<y_i\leq x_i\right\}\] 
	by $H_{e_1,e_2}(C)$. Finally, we use $\mathbf{1}$ to denote the element of $\NN^t$ whose coordinates are all $1$.
\end{notation}

\begin{definition}
	Let $e_1\in \NN$. Let $C$ be a subset of $\frac{1}{p^{e_1}}\NN^t$. 
	We say that $x\in C$ is a border point of $C$ if $x+\frac{1}{p^{e_1}}\mathbf{1}\not\in C$. 
	We denote by $\partial C$ the set of all border points in $C$. 
\end{definition}

\begin{notation} 
Given a subset $A$ of $\RR^t$ and $v\in \RR^t$, $A+v$ denotes the set $\{a+v \;| \; a \in A \}$.
%Given $A$ and $B$ subset of $\RR^t$, $A+B$ denotes the Minkowski sum, that is, $A+B=\{a+b \;| \; a \in A \; \mathrm{and}\; b\in B\}$.
Let $\underline{I}=I_1,\ldots, I_t\subseteq R$ be a sequence of ideals, $J_\bullet=\{J_{p^e}\}_{e \in \NN}$ be a $p$-family of ideals in $R$ such that $I_1,\ldots, I_t\subseteq \sqrt{J_1}$, and $\mu=\max\{\mu(I_1),\ldots,\mu(I_t)\}$. 
	Consider $e_1, e_2\in \NN$. 
	For each $n\in\{1,...,t\}$, let $\ell_n=\min\{\ell\mid I_n^{\ell}\subseteq J_1\}$.
	Then,
	\begin{itemize}
	\item$\mathcal{B}_n(\underline{I})_{e_1}=\frac{1}{p^{e_1}}\NN^{t-1}\cap\left(\prod_{i=1}^{n-1}[0,\mu(I_i)\ell_i]\times\prod_{i=n+1}^t[0,\mu(I_i)\ell_i]\right)$
	 \item$\mathcal{B}(\underline{I})_{e_1}=\frac{1}{p^{e_1}}\NN^t\cap\left(\bigcup_{j=1}^t\left(\prod_{i=1}^{j-1}[0,\mu(I_i)\ell_i]\times \{0\}\times\prod_{i=j+1}^t[0,\mu(I_i)\ell_i]\right)\right)$
	\item$\mathcal{R}_{e_1,e_2}=H_{e_1,e_2}(\frac{1}{p^{e_1}}\V_{\underline{I}}^{J_\bullet}(p^{e_1})),$ and
	\item $\mathcal{L}_{e_1,e_2}=
		H_{e_1,e_2}\left(  \bigcup^\mu_{j=0} \left(\partial \left(\frac{1}{p^{e_1}}\V_{\underline{I}}^{J_\bullet}(p^{e_1})\cup \mathcal{B}(\underline{I})_{e_1}\right)+\frac{j}{p^{e_1}}\mathbf{1}\right)\right).$
	\end{itemize}
\end{notation}

Roughly speaking, $\mathcal{R}_{e_1,e_2}$ is the result of filling the set $\frac{1}{p^{e_1}}\V_{\underline{I}}^{J_\bullet}(p^{e_1})$ when considered as a subset of $\frac{1}{p^{e_1+e_2}}\NN^t$.
Similarly we can think of $\mathcal{L}_{e_1,e_2}$ as the result of filling the subset of  $\frac{1}{p^{e_1+e_2}}\NN^t$ consisting of points in  $\frac{1}{p^{e_1}}\NN^t$ that are in the line segments joining $x\in\partial \left(\frac{1}{p^{e_1}}\V_{\underline{I}}^{J_\bullet}(p^{e_1})\cup \mathcal{B}(\underline{I})_{e_1}\right)$ with $x+\frac{1}{p^{e_1}}\mu\mathbf{1}$.

We now show an example that illustrates the regions  previously described.

\begin{example}
Suppose that $R=K[[x,y]]$ where $K$ is a field of characteristic $p=2$. Consider $\m=(x,y)$, the maximal ideal of $R$. Let $\underline{I}=xR,(y^2+x)R$.
	Then we have that
	\[\V_{\underline{I}}^{\m}(p^{e_1})=\left(\left([0,2^{e_1}-1]\times \left[0,\frac{2^{e_1}-2}{2}\right] \right)\cup \left(\left[0,\frac{2^{e_1}-2}{2}\right]\times \left( \frac{2^{e_1}-2}{2},2^{e_1}-1\right]\right)\right)\cap \mathbb{N}^2 .\]
	Note that $\mu=1$ and $\ell_1=\ell_2=1$.
	It follows that 
	\begin{align*}
	\partial \left(\frac{1}{p^{e_1}}\V_{\underline{I}}^{\m}(p^{e_1})\cup\mathcal{B}(\underline{I})_{e_1}\right)=&\frac{1}{2^{e_1}}\bigg(\bigg(\left(\{2^{e_1}-1\}\times\left[0,\frac{2^{e_1}-2}{2}\right]\right)\\&\cup\left(\left[\frac{2^{e_1}-2}{2},2^{e_1}-1\right]\times \left\{\frac{2^{e_1}-2}{2}\right
	\}\right)\\
	&\cup\left( \left\{\frac{2^{e_1}-2}{2}\right\}\times \left( \frac{2^{e_1}-2}{2},2^{e_1}-1\right]\right)\\
	&\cup \left( \left[0,\frac{2^{e_1}-2}{2}\right]\times \left\{2^{e_1}-1\right\}\right)\\
	&\cup\{(2^{e_1},0),(0, 2^{e_1})\}\bigg)\cap\mathbb{N}^2 \bigg).
	\end{align*}
	Additionally, we have the following equalities 
	\begin{itemize}
	\item $\mathcal{R}_{e_1,e_2}=\left(\bigcup_{x\in\frac{1}{p^{e_1}}\V_{\underline{I}}^{\m}(p^{e_1})}[0,x_1]\times\cdots\times[0,x_t]\right)\cap\frac{1}{p^{e_1+e_2}}\NN^2,$
	\item $\mathcal{L}_{e_1,e_2}=
		H_{e_1,e_2}\left(  \bigcup^1_{j=0} \left(\partial \left(\frac{1}{p^{e_1}}\V_{\underline{I}}^{J_\bullet}(p^{e_1})\cup \mathcal{B}(\underline{I})_{e_1}\right)+\frac{j}{p^{e_1}}\mathbf{1}\right)\right).$
	\end{itemize}
	
	The following figures show the regions of interest in the case $e_1=2, e_2=1, \mu=1$.
	The blue circles represent $\frac{1}{p^{e_1}}\V_{\underline{I}}^\m(p^{e_1})$.
	The blue circles together with the red squares represent $\mathcal{R}_{e_1,e_2}$.
	The border points of $\frac{1}{p^{e_1}}\V_{\underline{I}}^\m(p^{e_1})\cup \mathcal{B}(\underline{I})_{e_1}$ are represented by the green triangles.
	The orange stars represent the elements of the set $\mathcal{L}_{e_1,e_2}$.

	
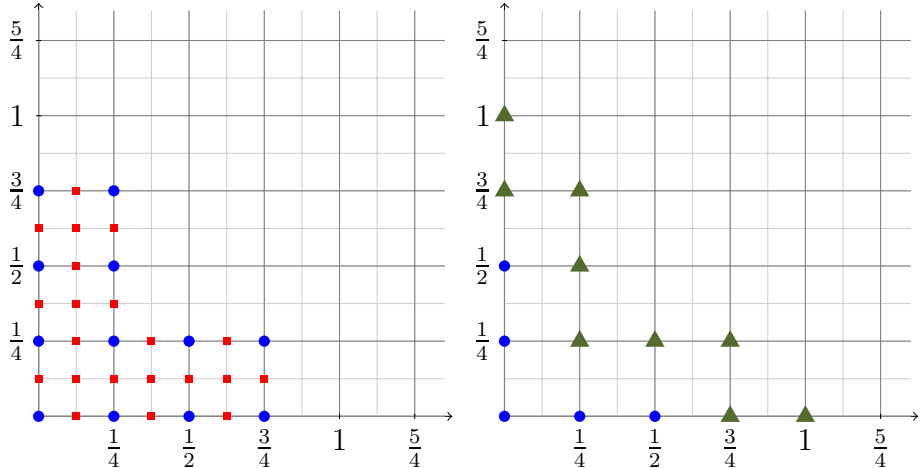
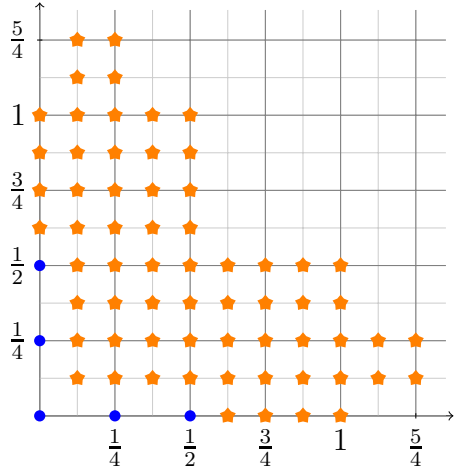
\begin{center}
	%original and R
	\begin{tikzpicture}
		\draw [->](0,0) --(5.5,0);
		\draw [->](0,0) --(0,5.5);
		\draw[step=.5 cm, gray!40, very thin](0,0) grid (5.4,5.4);
		\draw[step=1 cm, gray](0,0) grid (5.4,5.4);

	%axes numbers
		\foreach \x/\xtext in {.25/\frac{1}{4},.5/\frac{1}{2},.75/\frac{3}{4},1, 1.25/\frac{5}{4}}	
		\draw (\x*4 cm,1pt) -- (\x*4 cm,-1pt) node[anchor=north] {$\xtext$};
		\foreach \y/\ytext in {.25/\frac{1}{4},.5/\frac{1}{2},.75/\frac{3}{4},1,1.25/\frac{5}{4}}
		\draw (1pt,\y*4 cm) -- (-1pt,\y*4 cm) node[anchor=east] {$\ytext$};

	%nodes
		%1cm Grid
		%column 0
		\node[circle, fill=blue, inner sep=1.5pt] at (0,0){};
		\node[circle, fill=blue, inner sep=1.5pt] at (0,1){};
		\node[circle, fill=blue, inner sep=1.5pt] at (0,2){};
		\node[circle, fill=blue, inner sep=1.5pt] at (0,3){};
		%\node[circle, fill=blue, inner sep=1.5pt] at (0,4){};
	
		%column 1
		\node[circle, fill=blue, inner sep=1.5pt] at (1,0){};
		\node[circle, fill=blue, inner sep=1.5pt] at (1,1){};
		\node[circle, fill=blue, inner sep=1.5pt] at (1,2){};
		\node[circle, fill=blue, inner sep=1.5pt] at (1 ,3){};
		%\node[circle, fill=blue, inner sep=1.5pt] at (1,4){};
	
		%column 2
		\node[circle, fill=blue, inner sep=1.5pt] at (2 ,0){};
		\node[circle, fill=blue, inner sep=1.5pt] at (2,1){};
		%\node[circle, fill=blue, inner sep=1.5pt] at (2,2){};
		%\node[circle, fill=blue, inner sep=1.5pt] at (2,3){};
		%\node[circle, fill=blue, inner sep=1.5pt] at (2,4){};
	
		%column 3
		\node[circle, fill=blue, inner sep=1.5pt] at (3,0){};
		\node[circle, fill=blue, inner sep=1.5pt] at (3,1){};
		%\node[circle, fill=blue, inner sep=1.5pt] at (3,2){};
		%\node[circle, fill=blue, inner sep=1.5pt] at (3,3){};
		%\node[circle, fill=blue, inner sep=1.5pt] at (3,4){};
	
		%column 4
		%\node[circle, fill=blue, inner sep=1.5pt] at (4,0){};
		%\node[circle, fill=blue, inner sep=1.5pt] at (4,1){};
		%\node[circle, fill=blue, inner sep=1.5pt] at (4,2){};
		%\node[circle, fill=blue, inner sep=1.5pt] at (4,3){};
		%\node[circle, fill=blue, inner sep=1.5pt] at (4,4){};
	
		%.5cm Grid
		%column 0
		%\node[fill=red, inner sep=1.5pt] at (0,0){};
		\node[fill=red, inner sep=1.5pt] at (0,0.5){};
		%\node[fill=red, inner sep=1.5pt] at (0,1){};
		\node[fill=red, inner sep=1.5pt] at (0,1.5){};
		%\node[fill=red, inner sep=1.5pt] at (0,2){};
		\node[fill=red, inner sep=1.5pt] at (0,2.5){};
		%\node[fill=red, inner sep=1.5pt] at (0,3){};
		%\node[fill=red, inner sep=1.5pt] at (0,3.5){};
		%\node[fill=red, inner sep=1.5pt] at (0,4){};
	
		%column 0.5
		\node[fill=red, inner sep=1.5pt] at (0.5,0){};
		\node[fill=red, inner sep=1.5pt] at (0.5,0.5){};
		\node[fill=red, inner sep=1.5pt] at (0.5,1){};
		\node[fill=red, inner sep=1.5pt] at (0.5,1.5){};
		\node[fill=red, inner sep=1.5pt] at (0.5,2){};
		\node[fill=red, inner sep=1.5pt] at (0.5,2.5){};
		\node[fill=red, inner sep=1.5pt] at (0.5,3){};
		%\node[fill=red, inner sep=1.5pt] at (0.5,3.5){};
		%\node[fill=red, inner sep=1.5pt] at (0.5,4){};
	
		%column 1
		%\node[fill=red, inner sep=1.5pt] at (1,0){};
		\node[fill=red, inner sep=1.5pt] at (1,0.5){};
		%\node[fill=red, inner sep=1.5pt] at (1,1){};
		\node[fill=red, inner sep=1.5pt] at (1,1.5){};
		%\node[fill=red, inner sep=1.5pt] at (1,2){};
		\node[fill=red, inner sep=1.5pt] at (1,2.5){};
		%\node[fill=red, inner sep=1.5pt] at (1,3){};
		%\node[fill=red, inner sep=1.5pt] at (1,3.5){};
		%\node[fill=red, inner sep=1.5pt] at (1,4){};
	
		%column 1.5
		\node[fill=red, inner sep=1.5pt] at (1.5,0){};
		\node[fill=red, inner sep=1.5pt] at (1.5,0.5){};
		\node[fill=red, inner sep=1.5pt] at (1.5,1){};
		%\node[fill=red, inner sep=1.5pt] at (1.5,1.5){};
		%\node[fill=red, inner sep=1.5pt] at (1.5,2){};
		%\node[fill=red, inner sep=1.5pt] at (1.5,2.5){};
		%\node[fill=red, inner sep=1.5pt] at (1.5,3){};
		%\node[fill=red, inner sep=1.5pt] at (1.5,3.5){};
		%\node[fill=red, inner sep=1.5pt] at (1.5,4){};
	
		%column 2
		%\node[fill=red, inner sep=1.5pt] at (2,0){};
		\node[fill=red, inner sep=1.5pt] at (2,0.5){};
		%\node[fill=red, inner sep=1.5pt] at (2,1){};
		%\node[fill=red, inner sep=1.5pt] at (2,1.5){};
		%\node[fill=red, inner sep=1.5pt] at (2,2){};
		%\node[fill=red, inner sep=1.5pt] at (2,2.5){};
		%\node[fill=red, inner sep=1.5pt] at (2,3){};
		%\node[fill=red, inner sep=1.5pt] at (2,3.5){};
		%\node[fill=red, inner sep=1.5pt] at (2,4){};
	
		%column 2.5
		\node[fill=red, inner sep=1.5pt] at (2.5,0){};
		\node[fill=red, inner sep=1.5pt] at (2.5,0.5){};
		\node[fill=red, inner sep=1.5pt] at (2.5,1){};
		%\node[fill=red, inner sep=1.5pt] at (2.5,1.5){};
		%\node[fill=red, inner sep=1.5pt] at (2.5,2){};
		%\node[fill=red, inner sep=1.5pt] at (2.5,2.5){};
		%\node[fill=red, inner sep=1.5pt] at (2.5,3){};
		%\node[fill=red, inner sep=1.5pt] at (2.5,3.5){};
		%\node[fill=red, inner sep=1.5pt] at (2.5,4){};
	
		%column 3
		%\node[fill=red, inner sep=1.5pt] at (3,0){};
		\node[fill=red, inner sep=1.5pt] at (3,0.5){};
		%\node[fill=red, inner sep=1.5pt] at (3,1){};
		%\node[fill=red, inner sep=1.5pt] at (3,1.5){};
		%\node[fill=red, inner sep=1.5pt] at (3,2){};
		%\node[fill=red, inner sep=1.5pt] at (3,2.5){};
		%\node[fill=red, inner sep=1.5pt] at (3,3){};
		%\node[fill=red, inner sep=1.5pt] at (3,3.5){};
		%\node[fill=red, inner sep=1.5pt] at (3,4){};
	
		%column 3.5
		%\node[fill=red, inner sep=1.5pt] at (3.5,0){};
		%\node[fill=red, inner sep=1.5pt] at (3.5,0.5){};
		%\node[fill=red, inner sep=1.5pt] at (3.5,1){};
		%\node[fill=red, inner sep=1.5pt] at (3.5,1.5){};
		%\node[fill=red, inner sep=1.5pt] at (3.5,2){};
		%\node[fill=red, inner sep=1.5pt] at (3.5,2.5){};
		%\node[fill=red, inner sep=1.5pt] at (3.5,3){};
		%\node[fill=red, inner sep=1.5pt] at (3.5,3.5){};
		%\node[fill=red, inner sep=1.5pt] at (3.5,4){};
	
		%column 4
		%\node[fill=red, inner sep=1.5pt] at (4,0){};
		%\node[fill=red, inner sep=1.5pt] at (4,0.5){};
		%\node[fill=red, inner sep=1.5pt] at (4,1){};
		%\node[fill=red, inner sep=1.5pt] at (4,1.5){};
		%\node[fill=red, inner sep=1.5pt] at (4,2){};
		%\node[fill=red, inner sep=1.5pt] at (4,2.5){};
		%\node[fill=red, inner sep=1.5pt] at (4,3){};
		%\node[fill=red, inner sep=1.5pt] at (4,3.5){};
		%\node[fill=red, inner sep=1.5pt] at (4,4){};

	\end{tikzpicture}
	%border points 
	\begin{tikzpicture}
		\definecolor{darkgreen}{rgb}{.33,.42,.18}
		\draw [->](0,0) --(5.5,0);
		\draw [->](0,0) --(0,5.5);
		\draw[step=.5 cm, gray!40, very thin](0,0) grid (5.4,5.4);
		\draw[step=1 cm, gray](0,0) grid (5.4,5.4);
	
	%axes numbers
		\foreach \x/\xtext in {.25/\frac{1}{4},.5/\frac{1}{2},.75/\frac{3}{4},1, 1.25/\frac{5}{4}}	
		\draw (\x*4 cm,1pt) -- (\x*4 cm,-1pt) node[anchor=north] {$\xtext$};
		\foreach \y/\ytext in {.25/\frac{1}{4},.5/\frac{1}{2},.75/\frac{3}{4},1,1.25/\frac{5}{4}}
		\draw (1pt,\y*4 cm) -- (-1pt,\y*4 cm) node[anchor=east] {$\ytext$};
	%nodes
		%1cm Grid
		%column 0
		\node[circle, fill=blue, inner sep=1.5pt] at (0,0){};
		\node[circle, fill=blue, inner sep=1.5pt] at (0,1){};
		\node[circle, fill=blue, inner sep=1.5pt] at (0,2){};
		\node[regular polygon, regular polygon sides=3, fill=darkgreen, inner sep=1.5pt] at (0,3){};
		\node[regular polygon, regular polygon sides=3, fill=darkgreen, inner sep=1.5pt] at (0,4){};
	
		%column 1
		\node[circle, fill=blue, inner sep=1.5pt] at (1,0){};
		\node[regular polygon, regular polygon sides=3, fill=darkgreen, inner sep=1.5pt] at (1,1){};
		\node[regular polygon, regular polygon sides=3, fill=darkgreen, inner sep=1.5pt] at (1,2){};
		\node[regular polygon, regular polygon sides=3, fill=darkgreen, inner sep=1.5pt] at (1 ,3){};
		%\node[fill=blue, inner sep=1.5pt] at (1,4){};
	
		%column 2
		\node[circle, fill=blue, inner sep=1.5pt] at (2 ,0){};
		\node[regular polygon, regular polygon sides=3, fill=darkgreen, inner sep=1.5pt] at (2,1){};
		%\node[fill=blue, inner sep=1.5pt] at (2,2){};
		%\node[fill=blue, inner sep=1.5pt] at (2,3){};
		%\node[fill=blue, inner sep=1.5pt] at (2,4){};
	
		%column 3
		\node[regular polygon, regular polygon sides=3, fill=darkgreen, inner sep=1.5pt] at (3,0){};
		\node[regular polygon, regular polygon sides=3, fill=darkgreen, inner sep=1.5pt] at (3,1){};
		%\node[fill=blue, inner sep=1.5pt] at (3,2){};
		%\node[fill=blue, inner sep=1.5pt] at (3,3){};
		%\node[fill=blue, inner sep=1.5pt] at (3,4){};
	
		%column 4
		\node[regular polygon, regular polygon sides=3, fill=darkgreen, inner sep=1.5pt] at (4,0){};
	\end{tikzpicture}
	\end{center}
	\begin{center}
	% L
	\begin{tikzpicture}
		\definecolor{darkgreen}{rgb}{.33,.42,.18}
		\draw [->](0,0) --(5.5,0);
		\draw [->](0,0) --(0,5.5);
		\draw[step=.5 cm, gray!40, very thin](0,0) grid (5.4,5.4);
		\draw[step=1 cm, gray](0,0) grid (5.4,5.4);
	
	%axes numbers
		\foreach \x/\xtext in {.25/\frac{1}{4},.5/\frac{1}{2},.75/\frac{3}{4},1, 1.25/\frac{5}{4}}	
		\draw (\x*4 cm,1pt) -- (\x*4 cm,-1pt) node[anchor=north] {$\xtext$};
		\foreach \y/\ytext in {.25/\frac{1}{4},.5/\frac{1}{2},.75/\frac{3}{4},1,1.25/\frac{5}{4}}
		\draw (1pt,\y*4 cm) -- (-1pt,\y*4 cm) node[anchor=east] {$\ytext$};
	%nodes
		%1cm Grid
		%column 0
		\node[circle, fill=blue, inner sep=1.5pt] at (0,0){};
		\node[circle, fill=blue, inner sep=1.5pt] at (0,1){};
		\node[circle, fill=blue, inner sep=1.5pt] at (0,2){};
		%\node[regular polygon, regular polygon sides=3, fill=darkgreen, inner sep=1.5pt] at (0,3){};
		%\node[fill=blue, inner sep=1.5pt] at (0,4){};
	
		%column 1
		\node[circle, fill=blue, inner sep=1.5pt] at (1,0){};
		%\node[regular polygon, regular polygon sides=3, fill=darkgreen, inner sep=1.5pt] at (1,1){};
		%\node[regular polygon, regular polygon sides=3, fill=darkgreen, inner sep=1.5pt] at (1,2){};
		%\node[regular polygon, regular polygon sides=3, fill=darkgreen, inner sep=1.5pt] at (1 ,3){};
		%\node[fill=blue, inner sep=1.5pt] at (1,4){};
	
		%column 2
		\node[circle, fill=blue, inner sep=1.5pt] at (2 ,0){};
		%\node[regular polygon, regular polygon sides=3, fill=darkgreen, inner sep=1.5pt] at (2,1){};
		%\node[fill=blue, inner sep=1.5pt] at (2,2){};
		%\node[fill=blue, inner sep=1.5pt] at (2,3){};
		%\node[fill=blue, inner sep=1.5pt] at (2,4){};
	
		%column 3
		%\node[regular polygon, regular polygon sides=3, fill=darkgreen, inner sep=1.5pt] at (3,0){};
		%\node[regular polygon, regular polygon sides=3, fill=darkgreen, inner sep=1.5pt] at (3,1){};
		%\node[fill=blue, inner sep=1.5pt] at (3,2){};
		%\node[fill=blue, inner sep=1.5pt] at (3,3){};
		%\node[fill=blue, inner sep=1.5pt] at (3,4){};
	
		%column 4
		%\node[fill=blue, inner sep=1.5pt] at (4,0){};
		%\node[fill=blue, inner sep=1.5pt] at (4,1){};
		%\node[fill=blue, inner sep=1.5pt] at (4,2){};
		%\node[fill=blue, inner sep=1.5pt] at (4,3){};
		%\node[fill=blue, inner sep=1.5pt] at (4,4){};
	
		%.5cm Grid
		%column 0
		%\node[fill=red, inner sep=1.5pt] at (0,0){};
		%\node[fill=red, inner sep=1.5pt] at (0,0.5){};
		%\node[fill=red, inner sep=1.5pt] at (0,1){};
		%\node[star, fill=orange, inner sep=1.5pt] at (0,1.5){};
		%\node[star, fill=orange, inner sep=1.5pt] at (0,2){};
		\node[star, fill=orange, inner sep=1.5pt] at (0,2.5){};
		\node[star, fill=orange, inner sep=1.5pt] at (0,3){};
		\node[star, fill=orange, inner sep=1.5pt] at (0,3.5){};
		\node[star, fill=orange, inner sep=1.5pt] at (0,4){};
	
		%column 0.5
		%\node[fill=red, inner sep=1.5pt] at (0.5,0){};
		\node[star, fill=orange, inner sep=1.5pt] at (0.5,0.5){};
		\node[star, fill=orange, inner sep=1.5pt] at (0.5,1){};
		\node[star, fill=orange, inner sep=1.5pt] at (0.5,1.5){};
		\node[star, fill=orange, inner sep=1.5pt] at (0.5,2){};
		\node[star, fill=orange, inner sep=1.5pt] at (0.5,2.5){};
		\node[star, fill=orange, inner sep=1.5pt] at (0.5,3){};
		\node[star, fill=orange, inner sep=1.5pt] at (0.5,3.5){};
		\node[star, fill=orange, inner sep=1.5pt] at (0.5,4){};
		\node[star, fill=orange, inner sep=1.5pt] at (0.5,4.5){};
		\node[star, fill=orange, inner sep=1.5pt] at (0.5,5){};
	
		%column 1
		%\node[fill=red, inner sep=1.5pt] at (1,0){};
		\node[star, fill=orange, inner sep=1.5pt] at (1,0.5){};
		\node[star, fill=orange, inner sep=1.5pt] at (1,1){};
		\node[star, fill=orange, inner sep=1.5pt] at (1,1.5){};
		\node[star, fill=orange, inner sep=1.5pt] at (1,2){};
		\node[star, fill=orange, inner sep=1.5pt] at (1,2.5){};
		\node[star, fill=orange, inner sep=1.5pt] at (1,3){};
		\node[star, fill=orange, inner sep=1.5pt] at (1,3.5){};
		\node[star, fill=orange, inner sep=1.5pt] at (1,4){};
		\node[star, fill=orange, inner sep=1.5pt] at (1,4.5){};
		\node[star, fill=orange, inner sep=1.5pt] at (1,5){};
	
		%column 1.5
		%\node[fill=red, inner sep=1.5pt] at (1.5,0){};
		\node[star, fill=orange, inner sep=1.5pt] at (1.5,0.5){};
		\node[star, fill=orange, inner sep=1.5pt] at (1.5,1){};
		\node[star, fill=orange, inner sep=1.5pt] at (1.5,1.5){};
		\node[star, fill=orange, inner sep=1.5pt] at (1.5,2){};
		\node[star, fill=orange, inner sep=1.5pt] at (1.5,2.5){};
		\node[star, fill=orange, inner sep=1.5pt] at (1.5,3){};
		\node[star, fill=orange, inner sep=1.5pt] at (1.5,3.5){};
		\node[star, fill=orange, inner sep=1.5pt] at (1.5,4){};
		%\node[star, fill=orange, inner sep=1.5pt] at (1.5,4.5){};
		%\node[star, fill=orange, inner sep=1.5pt] at (1.5,5){};
		
		%column 2
		%\node[fill=red, inner sep=1.5pt] at (2,0){};
		\node[star, fill=orange, inner sep=1.5pt] at (2,0.5){};
		\node[star, fill=orange, inner sep=1.5pt] at (2,1){};
		\node[star, fill=orange, inner sep=1.5pt] at (2,1.5){};
		\node[star, fill=orange, inner sep=1.5pt] at (2,2){};
		\node[star, fill=orange, inner sep=1.5pt] at (2,2.5){};
		\node[star, fill=orange, inner sep=1.5pt] at (2,3){};
		\node[star, fill=orange, inner sep=1.5pt] at (2,3.5){};
		\node[star, fill=orange, inner sep=1.5pt] at (2,4){};
		%\node[star, fill=orange, inner sep=1.5pt] at (2,4.5){};
		%\node[star, fill=orange, inner sep=1.5pt] at (2,5){};
		
		%column 2.5
		\node[star, fill=orange, inner sep=1.5pt] at (2.5,0){};
		\node[star, fill=orange, inner sep=1.5pt] at (2.5,0.5){};
		\node[star, fill=orange, inner sep=1.5pt] at (2.5,1){};
		\node[star, fill=orange, inner sep=1.5pt] at (2.5,1.5){};
		\node[star, fill=orange, inner sep=1.5pt] at (2.5,2){};
		%\node[star, fill=orange, inner sep=1.5pt] at (2.5,2.5){};
		%\node[star, fill=orange, inner sep=1.5pt] at (2.5,3){};
		%\node[star, fill=orange, inner sep=1.5pt] at (2.5,3.5){};
		%\node[star, fill=orange, inner sep=1.5pt] at (2.5,4){};
		%\node[star, fill=orange, inner sep=1.5pt] at (2.5,4.5){};
		%\node[star, fill=orange, inner sep=1.5pt] at (2.5,5){};
		
		%column 3
		\node[star, fill=orange, inner sep=1.5pt] at (3,0){};
		\node[star, fill=orange, inner sep=1.5pt] at (3,0.5){};
		\node[star, fill=orange, inner sep=1.5pt] at (3,1){};
		\node[star, fill=orange, inner sep=1.5pt] at (3,1.5){};
		\node[star, fill=orange, inner sep=1.5pt] at (3,2){};
		%\node[star, fill=orange, inner sep=1.5pt] at (3,2.5){};
		%\node[star, fill=orange, inner sep=1.5pt] at (3,3){};
		%\node[star, fill=orange, inner sep=1.5pt] at (3,3.5){};
		%\node[star, fill=orange, inner sep=1.5pt] at (3,4){};
		%\node[star, fill=orange, inner sep=1.5pt] at (3,4.5){};
		%\node[star, fill=orange, inner sep=1.5pt] at (3,5){};
	
		%column 3.5
		\node[star, fill=orange, inner sep=1.5pt] at (3.5,0){};
		\node[star, fill=orange, inner sep=1.5pt] at (3.5,0.5){};
		\node[star, fill=orange, inner sep=1.5pt] at (3.5,1){};
		\node[star, fill=orange, inner sep=1.5pt] at (3.5,1.5){};
		\node[star, fill=orange, inner sep=1.5pt] at (3.5,2){};
		%\node[star, fill=orange, inner sep=1.5pt] at (3.5,2.5){};
		%\node[star, fill=orange, inner sep=1.5pt] at (3.5,3){};
		%\node[fill=red, inner sep=1.5pt] at (3.5,3.5){};
		%\node[fill=red, inner sep=1.5pt] at (3.5,4){};
	
		%column 4
		\node[star, fill=orange, inner sep=1.5pt] at (4,0){};
		\node[star, fill=orange, inner sep=1.5pt] at (4,0.5){};
		\node[star, fill=orange, inner sep=1.5pt] at (4,1){};
		\node[star, fill=orange, inner sep=1.5pt] at (4,1.5){};
		\node[star, fill=orange, inner sep=1.5pt] at (4,2){};
		%\node[star, fill=orange, inner sep=1.5pt] at (4,2.5){};
		%\node[star, fill=orange, inner sep=1.5pt] at (4,3){};
		%\node[fill=red, inner sep=1.5pt] at (4,3.5){};
		%\node[fill=red, inner sep=1.5pt] at (4,4){};
		
		%column 4.5
		%\node[fill=red, inner sep=1.5pt] at (4.5,0){};
		\node[star, fill=orange, inner sep=1.5pt] at (4.5,0.5){};
		\node[star, fill=orange, inner sep=1.5pt] at (4.5,1){};
		%\node[star, fill=orange, inner sep=1.5pt] at (4.5,1.5){};
		%\node[star, fill=orange, inner sep=1.5pt] at (4.5,2){};
		%\node[star, fill=orange, inner sep=1.5pt] at (4.5,2.5){};
		%\node[star, fill=orange, inner sep=1.5pt] at (4.5,3){};
		%\node[fill=red, inner sep=1.5pt] at (4.5,3.5){};
		%\node[fill=red, inner sep=1.5pt] at (4.5,4){};
	
		%column 5
		%\node[fill=red, inner sep=1.5pt] at (5,0){};
		\node[star, fill=orange, inner sep=1.5pt] at (5,0.5){};
		\node[star, fill=orange, inner sep=1.5pt] at (5,1){};
		%\node[star, fill=orange, inner sep=1.5pt] at (5,1.5){};
		%\node[star, fill=orange, inner sep=1.5pt] at (5,2){};
		%\node[star, fill=orange, inner sep=1.5pt] at (5,2.5){};
		%\node[star, fill=orange, inner sep=1.5pt] at (5,3){};
		%\node[fill=red, inner sep=1.5pt] at (5,3.5){};
		%\node[fill=red, inner sep=1.5pt] at (5,4){};
	\end{tikzpicture}
\end{center}

We explore this discussion further in Example \ref{ex}.
\end{example}
\begin{lemma}\label{lemma-cardinal}
	Let $\underline{I}=I_1,\ldots, I_t\subseteq R$ be a sequence of ideals, and $J_\bullet=\{J_{p^e}\}_{e \in \NN}$ be a $p$-family of ideals in $R$ such that $I_1,\ldots, I_t\subseteq \sqrt{J_1}$. Then, for every positive integer $e_1$, we have that
	\[\left\vert\partial \left(\frac{1}{p^{e_1}}\V_{\underline{I}}^{J_\bullet}(p^{e_1})\cup \mathcal{B}(\underline{I})_{e_1}\right)\right\vert\leq
	p^{e_1(t-1)}\sum_{n=1}^t\left(\prod_{j=1}^{n-1}(\mu(I_j)\ell_j+1)\prod_{j=n+1}^t(\mu(I_j)\ell_j+1)\right).\]
\end{lemma}

\begin{proof}
	Let $e_1$ be a positive integer. %and let $\underline{I}=I_1,\ldots, I_t\subseteq R$ be a sequence of ideals. 
	Let $\phi:\partial \left(\frac{1}{p^{e_1}}\V_{\underline{I}}^{J_\bullet}(p^{e_1})\cup \mathcal{B}(\underline{I})_{e_1}\right)\rightarrow \bigcup_{j=1}^t\left(\mathcal{B}_j(\underline{I})_{e_1}\times \{j\}\right)$ the map defined by
	 \[\phi(x_1,\ldots.,x_t)=(\widehat{y}^s,s)\]
	where 
	$$y=(x_1,\ldots.,x_t)-\min\{x_i:i\in \{1,\ldots,t\}\}\mathbf{1}$$ 
	and 
	\[s=\min\{i\in\{1,\ldots,t\}:x_i=\min\{x_j:j\in \{1,\ldots,t\}\}\}.\]
	Notice that, if $(x_1,\ldots,x_t)\in \frac{1}{p^{e_1}}\V_{\underline{I}}^{J_\bullet}(p^{e_1})$, we have that $y\in \frac{1}{p^{e_1}}\V_{\underline{I}}^{J_\bullet}(p^{e_1})$ and $\widehat{y}^s\in\mathcal{B}_s(\underline{I})_{e_1}$ by Remark \ref{finiteness}.
	On the other hand, if $x=(x_1,\ldots,x_t)\in\mathcal{B}(\underline{I})_{e_1}$, then $\min\{x_i:i\in \{1,\ldots,t\}\})=0$ and $y=x$.
	 Hence $\widehat{y}^s\in \mathcal{B}_s(\underline{I})_{e_1}$.
	Thus $\phi$ is well-defined. 
	Now suppose $\phi(x_1,\ldots,x_n)=\phi(z_1,\ldots.,z_n)$. 
	It follows that  $(z_1,\ldots.,z_t)-z_s\mathbf{1} =(x_1,\ldots,x_t)-x_s\mathbf{1}$.
	We can assume without loss of generality that $z_s\geq x_s$. 
	Then, we have that $(z_1,\ldots,z_t)=(x_1,\ldots,x_t)+(z_s-x_s)\mathbf{1}$. 
	If $z_s>x_s$, then $z_i\geq x_i+\frac{1}{p^{e_1}}$ and $z_i>0$ for every $i\in \{1,\ldots,t\}$.
	Consequently, $(z_1,\ldots,z_t)\in \frac{1}{p^{e_1}}\V_{\underline{I}}^{J_\bullet}(p^{e_1})$ and 
	$(x_1,\ldots,x_t)+\frac{1}{p^{e_1}}\mathbf{1} \in \frac{1}{p^{e_1}}\V_{\underline{I}}^{J_\bullet}(p^{e_1})\cup\mathcal{B}(\underline{I})_{e_1}$, which contradicts that $(x_1,\ldots, x_t)\in \partial\left(\frac{1}{p^{e_1}}\V_{\underline{I}}^{J_\bullet}(p^{e_1})\cup\mathcal{B}(\underline{I})_{e_1}\right)$.
	Hence, $\phi$ is injective. 
	Therefore, we have that
	\[\left\vert\partial \left(\frac{1}{p^{e_1}}\V_{\underline{I}}^{J_\bullet}(p^{e_1})\cup \mathcal{B}(\underline{I})_{e_1}\right)\right\vert\leq
	p^{e_1(t-1)}\sum_{n=1}^t\left(\prod_{j=1}^{n-1}(\mu(I_j)\ell_j+1)\prod_{j=n+1}^t(\mu(I_j)\ell_j+1)\right).\]
\end{proof}

\begin{remark}\label{remarkH}
Let $e_1,e_2 \in \NN$, $C$ be a subset of $\frac{1}{p^{e_1}}\NN^t$, and $x$ be an element of $\frac{1}{p^{e_1+e_2}}\NN^t$. Suppose that $\frac{1}{p^{e_1}}\left\lceil p^{e_1}x\right \rceil \in C$. For every $i\in\{1,\ldots,t\}$ we have that
	\[ x_i+\frac{1}{p^{e_1}}=\frac{1}{p^{e_1}}\left(p^{e_1}x_i+1\right)> \frac{1}{p^{e_1}}\left \lceil p^{e_1}x_i\right \rceil.\]
	Thus, 
	\[\frac{1}{p^{e_1}}\left \lceil p^{e_1}x_i\right \rceil-\frac{1}{p^{e_1}}<x_i\leq
	 \frac{1}{p^{e_1}}\left \lceil p^{e_1}x_i\right \rceil.\]
Therefore, $x \in H_{e_1,e_2}(C)$.	 
\end{remark} 

We now start a series of lemmas towards proving Theorem \ref{MainThmExistence}. 
Roughly speaking, throughout the following lemmas we compare $\frac{1}{p^{e_1+e_2}}\V_{\underline{I}}^{J_\bullet}(p^{e_1+e_2})$ with $\frac{1}{p^{e_1}}\V_{\underline{I}}^{J_\bullet}(p^{e_1})$ while considering both as subsets of $\frac{1}{p^{e_1+e_2}}\NN^t$. To do this, we use $H_{e_1,e_2}$ to fill $\frac{1}{p^{e_1}}\V_{\underline{I}}^{J_\bullet}(p^{e_1})$, obtaining $\mathcal{R}_{e_1,e_2}$, and we use $\mathcal{L}_{e_1,e_2}$ to control the behavior of  $\frac{1}{p^{e_1+e_2}}\V_{\underline{I}}^{J_\bullet}(p^{e_1+e_2})\setminus \mathcal{R}_{e_1,e_2}$.

\begin{lemma}\label{lem2}
Let $\underline{I}=I_1,\ldots, I_t\subseteq R$ be a sequence of ideals, and $J_\bullet=\{J_{p^e}\}_{e \in \NN}$ be a $p$-family of ideals in $R$ such that $I_1,\ldots, I_t\subseteq \sqrt{J_1}$. We have that
	\[
	\frac{1}{p^{e_1+e_2}}\V_{\underline{I}}^{J_\bullet}(p^{e_1+e_2})\subseteq \mathcal{R}_{e_1,e_2}\cup\mathcal{L}_{e_1, e_2}.\]
\end{lemma}

\begin{proof}
Let $x=(x_1,\ldots, x_t)\in \frac{1}{p^{e_1+e_2}}\V_{\underline{I}}^{J_\bullet}(p^{e_1+e_2})$ be such that $x\not\in \mathcal{R}_{e_1,e_2}.$ By Remark \ref{finiteness}, $p^{e_1+e_2}x_n \leq\mu(I_n)\ell_n p^{e_1+e_2}$ for each $n\in\{1,\ldots,t\}$.
	Hence,  $p^{e_1}x_n \leq\mu(I_n)\ell_np^{e_1}$ and $\ceil{p^{e_1}x_n} \leq\mu(I_n)\ell_np^{e_1}$ for each $n\in\{1,\ldots,t\}$.
	Thus,  if $x_j=\operatorname{min}\{x_1,\ldots,x_t\}$, we have
	$$\frac{1}{p^{e_1}}\left(\left\lceil p^{e_1}x\right \rceil-\left \lceil p^{e_1}x_j\right\rceil\mathbf{1}\right)\in \mathcal{B}(\underline{I})_{e_1}.
	$$ Hence, $\left\{y\in \frac{1}{p^{e_1}}\ZZ\mid\frac{1}{p^{e_1}}\left \lceil p^{e_1}x\right \rceil-y\mathbf{1}\in \left(\frac{1}{p^{e_1}}\V_{\underline{I}}^{J_\bullet}(p^{e_1}) \cup \mathcal{B}(\underline{I})_{e_1}\right)\right\}$ is not empty.
	 
 Since $x \not\in\mathcal{R}_{e_1,e_2}$, we have that
	$\frac{1}{p^{e_1}}\left\lceil p^{e_1}x\right \rceil+y\mathbf{1} \not\in \frac{1}{p^{e_1}}\V_{\underline{I}}^{J_\bullet}(p^{e_1})$ 
	for every $y\in \frac{1}{p^{e_1}}\NN$ by Remark \ref{remarkH}. As a consequence, $\left\{y\in \frac{1}{p^{e_1}}\ZZ\mid\frac{1}{p^{e_1}}\left \lceil p^{e_1}x\right \rceil-y\mathbf{1}\in \left(\frac{1}{p^{e_1}}\V_{\underline{I}}^{J_\bullet}(p^{e_1}) \cup \mathcal{B}(\underline{I})_{e_1}\right)\right\}$ is bounded below by $0$.
 
 We take $a=\frac{1}{p^{e_1}}\left \lceil p^{e_1}x\right \rceil-r\mathbf{1}$, where $$r=\min\left\{y\in \frac{1}{p^{e_1}}\ZZ\mid\frac{1}{p^{e_1}}\left \lceil p^{e_1}x\right \rceil-y\mathbf{1}\in \left(\frac{1}{p^{e_1}}\V_{\underline{I}}^{J_\bullet}(p^{e_1}) \cup \mathcal{B}(\underline{I})_{e_1}\right)\right\}.$$
	We note that $a\in \partial \left( \frac{1}{p^{e_1}}\V_{\underline{I}}^{J_\bullet}(p^{e_1})\cup \mathcal{B}(\underline{I})_{e_1}\right)$. 
	
On the other hand, by Lemma \ref{DNBP1} with $s=p^{e_1}a_1,\ldots., p^{e_1}a_t$, we have
	\begin{align*}
	\underline{I}^{p^{e_2}(p^{e_1}a+\mu\mathbf{1})}&=I_1^{p^{e_2}(p^{e_1}a_1+\mu)}\cdots I_t^{p^{e_2}(p^{e_1}a_t+\mu)}\\ 
	&= I_1^{\mu p^{e_2}}(I_1^{[p^{e_2}]})^{p^{e_1}a_1}\cdots I_t^{\mu p^{e_2}}(I_t^{[p^{e_2}]})^{p^{e_1}a_t}\\
	&\subseteq I_1^{[p^{e_2}]}(I_1^{p^{e_1}a_1})^{[p^{e_2}]}\cdots I_t^{[p^{e_2}]}(I_t^{p^{e_1}a_t})^{[p^{e_2}]}\\
	&=(I_1^{p^{e_1}a_1+1})^{[p^{e_2}]}\cdots (I_t^{p^{e_1}a_t+1})^{[p^{e_2}]}\\
	&=(\underline{I}^{p^{e_1}a+\mathbf{1}})^{[p^{e_2}]}.
	\end{align*}
Since $a\in \partial \left( \frac{1}{p^{e_1}}\V_{\underline{I}}^{J_\bullet}(p^{e_1})\cup \mathcal{B}(\underline{I})_{e_1}\right)$, $p^{e_1}a+\mathbf{1} \not \in \V_{\underline{I}}^{J_\bullet}(p^{e_1})$. Then,
\begin{align*}
\underline{I}^{p^{e_2}(p^{e_1}a+\mu\mathbf{1})}&=(\underline{I}^{p^{e_1}a+\mathbf{1}})^{[p^{e_2}]}\\
&\subseteq J_{p^{e_1}}^{[p^{e_2}]}\\
&\subseteq J_{p^{e_1+e_2}}.
\end{align*}
Thus, $p^{e_2}(p^{e_1}a+\mu\mathbf{1}) \not \in \V_{\underline{I}}^{J_\bullet}(p^{e_1+e_2})$. Hence, there exists $k\in \{1,\ldots,t\}$ such that $p^{e_1+e_2}x_k \leq p^{e_2}(p^{e_1}a_k+\mu)$. This implies, $p^{e_1}x_k \leq \left \lceil p^{e_1}x_k\right \rceil -p^{e_1}r+\mu$, and so $\left \lceil p^{e_1}x_k\right \rceil \leq \left \lceil p^{e_1}x_k\right \rceil -p^{e_1}r+\mu$. Then, we have that $0\leq r\leq \frac{1}{p^{e_1}}\mu$. 
	 
Since $\frac{1}{p^{e_1}}\left \lceil p^{e_1}x\right \rceil=a+r\mathbf{1} \in  \bigcup^\mu_{j=0} \left(\partial \left(\frac{1}{p^{e_1}}\V_{\underline{I}}^{J_\bullet}(p^{e_1})\cup \mathcal{B}(\underline{I})_{e_1}\right)+\frac{j}{p^{e_1}}\mathbf{1}\right)$, it follows that $x\in H_{e_1,e_2}\left( \bigcup^\mu_{j=0} \left(\partial \left(\frac{1}{p^{e_1}}\V_{\underline{I}}^{J_\bullet}(p^{e_1})\cup \mathcal{B}(\underline{I})_{e_1}\right)+\frac{j}{p^{e_1}}\mathbf{1}\right)\right)=\mathcal{L}_{e_1, e_2}$ by Remark \ref{remarkH}. 
\end{proof}   

\begin{lemma}\label{lem3}
Let $\underline{I}=I_1,\ldots, I_t\subseteq R$ be a sequence of ideals, and $J_\bullet=\{J_{p^e}\}_{e \in \NN}$ be a $p$-family of ideals in $R$ such that $I_1,\ldots, I_t\subseteq \sqrt{J_1}$. For each $e_1\in \mathbb{N}$, there exists a subset $A_{e_1}$ of $\frac{1}{p^{e_1}}\NN^t$ such that  
	\begin{enumerate}
	\item$\frac{1}{p^{e_1}}\V_{\underline{I}}^{J_\bullet}(p^{e_1})\subseteq A_{e_1}$,
	\item $\frac{1}{p^{e_1+e_2}}\V_{\underline{I}}^{J_\bullet}(p^{e_1+e_2})\subseteq H_{e_1,e_2}(A_{e_1})$ for all $e_2\in \NN$, and
	\item
	$\lim_{e_1\to \infty}\frac{\vert A_{e_1}  \setminus\frac{1}{p^{e_1}}\V_{\underline{I}}^{J_\bullet}(p^{e_1})\vert}{p^{e_1t}}=0.$
	\end{enumerate}
\end{lemma}

\begin{proof}
	By Lemma \ref{lem2}, we have that
	\begin{align*}
	\frac{1}{p^{e_1+e_2}}\V_{\underline{I}}^{J_\bullet}(p^{e_1+e_2})&\subseteq \mathcal{R}_{e_1,e_2}\cup \mathcal{L}_{e_1,e_2}\\
	&\subseteq H_{e_1,e_2}\left(\frac{1}{p^{e_1}}\V_{\underline{I}}^{J_\bullet}(p^{e_1})\bigcup\left(  \bigcup^\mu_{j=0} \left(\partial \left(\frac{1}{p^{e_1}}\V_{\underline{I}}^{J_\bullet}(p^{e_1})\cup \mathcal{B}(\underline{I})_{e_1}\right)+\frac{j}{p^{e_1}}\mathbf{1}\right)\right) \right)
	\end{align*}
	Let $A_{e_1}=\frac{1}{p^{e_1}}\V_{\underline{I}}^{J_\bullet}(p^{e_1})\bigcup\left( \bigcup^\mu_{j=0} \left(\partial \left(\frac{1}{p^{e_1}}\V_{\underline{I}}^{J_\bullet}(p^{e_1})\cup \mathcal{B}(\underline{I})_{e_1}\right)+\frac{j}{p^{e_1}}\mathbf{1}\right)\right) $ 
	Then, \[\left\vert A_{e_1}  \setminus \frac{1}{p^{e_1}}\V_{\underline{I}}^{J_\bullet}(p^{e_1})\right\vert\leq p^{e_1(t-1)}(\mu+1)\sum_{n=1}^t\left(\prod_{j=1}^{n-1}(\mu(I_j)\ell_j+1)\prod_{j=n+1}^t(\mu(I_j)\ell_j+1)\right)\]
	by Lemma \ref{lemma-cardinal}. 
	Hence, we have that
	\begin{align*}
	0&\leq\liminf_{e_1\to \infty}\frac{\vert A_{e_1}  \setminus \frac{1}{p^{e_1}}\V_{\underline{I}}^{J_\bullet}(p^{e_1})\vert}{p^{e_1t}}\\
	&\leq\limsup_{e_1\to \infty}\frac{\vert A_{e_1}  \setminus \frac{1}{p^{e_1}}\V_{\underline{I}}^{J_\bullet}(p^{e_1})\vert}{p^{e_1t}}\\
	&\leq\limsup_{e_1\to \infty}\frac{(\mu+1)\sum_{n=1}^t\left(\prod_{j=1}^{n-1}(\mu(I_j)\ell_j+1)\prod_{j=n+1}^t(\mu(I_j)\ell_j+1)\right)}{p^{e_1}}\\
	&=0.
	\end{align*}
	It follows that 
	\[\lim_{e_1\to \infty}\frac{\vert A_{e_1}  \setminus \frac{1}{p^{e_1}}\V_{\underline{I}}^{J_\bullet}(p^{e_1})\vert}{p^{e_1t}}=0.\]
\end{proof}

We are now ready to prove the main result of this section. This result implies  Theorem \ref{MainThmExistence} in the introduction, because $\{J^{[p^e]}\}_{e \in \NN}$ is a $p$-family of ideals.

\begin{theorem}\label{ThmLimitExistsSec}
Let $\underline{I}=I_1,\ldots, I_t\subseteq R$ be a sequence of ideals, and $J_\bullet=\{J_{p^e}\}_{e \in \NN}$ be a $p$-family of ideals in $R$ such that $I_1,\ldots, I_t\subseteq \sqrt{J_1}$. Then, $\lim_{e\to \infty}\frac{\vert \V_{\underline{I}}^{J_\bullet}(p^{e})\vert}{p^{et}}$ exists. 
\end{theorem}
\begin{proof}
	For each $e_1\in \NN$, let $A_{e_1}$ be as in Lemma \ref{lem3}. 	
	Then, for each $e_1,e_2\in\NN$, we have 
 $$\frac{1}{p^{e_1+e_2}}\V_{\underline{I}}^{J_\bullet}(p^{e_1+e_2})\subseteq H_{e_1,e_2}(A_{e_1}).$$ As a consequence, 
 $$\vert \V_{\underline{I}}^{J_\bullet}(p^{e_1+e_2})\vert\leq |H_{e_1,e_2}(A_{e_1})| \leq p^{e_2t}|A_{e_1}|.$$	
Since $\frac{1}{p^{e_1}}\V_{\underline{I}}^{J_\bullet}(p^{e_1}) \subseteq A_{e_1}$, $$A_{e_1}=\frac{1}{p^{e_1}}\V_{\underline{I}}^{J_\bullet}(p^{e_1}) \cup \left( A_{e_1} \setminus \frac{1}{p^{e_1}}\V_{\underline{I}}^{J_\bullet}(p^{e_1}) \right).$$ Hence, $$|A_{e_1}|=\left(\vert \V_{\underline{I}}^{J_\bullet}(p^{e_1})\vert+ \left\vert A_{e_1} \setminus \frac{1}{p^{e_1}}\V_{\underline{I}}^{J_\bullet}(p^{e_1})\right\vert\right).$$
It follows that
	\[\vert \V_{\underline{I}}^{J_\bullet}(p^{e_1+e_2})\vert\leq p^{e_2t}\left(\vert \V_{\underline{I}}^{J_\bullet}(p^{e_1})\vert+ \left\vert A_{e_1} \setminus \frac{1}{p^{e_1}}\V_{\underline{I}}^{J_\bullet}(p^{e_1})\right\vert\right).\]
	Dividing by $p^{e_1t+e_2t}$, we obtain
	\[\frac{\vert \V_{\underline{I}}^{J_\bullet}(p^{e_1+e_2})\vert}{p^{e_1t+e_2t}}\leq \frac{\vert \V_{\underline{I}}^{J_\bullet}(p^{e_1})\vert}{p^{e_1t}}+ \frac{\vert A_{e_1} \setminus \frac{1}{p^{e_1t}}\V_{\underline{I}}^{J_\bullet}(p^{e_1})\vert}{p^{e_1t}}.\]
	Thus, we have
	\[\limsup_{e\to \infty}\frac{\vert \V_{\underline{I}}^{J_\bullet}(p^{e})\vert}{p^{et}}=\limsup_{e_2\to \infty}\frac{\vert \V_{\underline{I}}^{J_\bullet}(p^{e_1+e_2})\vert}{p^{e_1t+e_2t}}\leq \frac{\vert \V_{\underline{I}}^{J_\bullet}(p^{e_1})\vert}{p^{e_1t}}+ \frac{\vert A_{e_1} \setminus \frac{1}{p^{e_1t}}\V_{\underline{I}}^{J_\bullet}(p^{e_1})\vert}{p^{e_1t}}.\]
	It follows that 
	\[\limsup_{e\to \infty}\frac{\vert \V_{\underline{I}}^{J_\bullet}(p^{e})\vert}{p^{et}}\leq\liminf_{e_1\to \infty} \frac{\vert \V_{\underline{I}}^{J_\bullet}(p^{e_1})\vert}{p^{e_1t}}+ \lim_{e_1\to\infty}\frac{\vert A_{e_1} \setminus \frac{1}{p^{e_1t}}\V_{\underline{I}}^{J_\bullet}(p^{e_1})\vert}{p^{e_1t}}=\liminf_{e_1\to \infty} \frac{\vert \V_{\underline{I}}^{J_\bullet}(p^{e_1})\vert}{p^{e_1t}}.\]
	Therefore, the $\lim_{e\to \infty}\frac{\vert \V_{\underline{I}}^{J_\bullet}(p^{e})\vert}{p^{et}}$ exists.  
\end{proof}

Given Theorem \ref{ThmLimitExistsSec}, we are able to define the $F$-volume of a sequence of ideals with respect to a $p$-family. We justify the choice of this name in Section \ref{SecFpure}, where we show that this number gives a volume of certain regions for $F$-pure rings.

\begin{definition} \label{DefFVol}
Let $\underline{I}=I_1,\ldots, I_t\subseteq R$ be a sequence of ideals, and $J_\bullet=\{J_{p^e}\}_{e \in \NN}$ be a $p$-family of ideals in $R$ such that $I_1,\ldots, I_t\subseteq \sqrt{J_1}$. We define the $F$-volume of the sequence  $\underline{I}$ with respect to the $p$-family $J_\bullet=\{J_{p^e}\}_{e \in \NN}$ by 
$$
\Vol^{J_\bullet}_F(\underline{I})=\lim\limits_{e\to\infty} \frac{1}{p^{et}} |\V_{\underline{I}}^{J_\bullet}(p^e)|.
$$
If $\underline{f}=f_1,\ldots,f_t$ is a sequence of elements of $R$ such that $f_1,\ldots, f_t\in\sqrt{J_1}$, we use $\Vol^{J_\bullet}_F(\underline{f})$ to denote $\Vol^{J_\bullet}_F(\underline{I})$ where $\underline{I}=f_1R,\ldots, f_tR$. In case that the $p$-family is $J_\bullet=\{J^{[p^e]}\}_{e \in \NN}$ where $J$ is an ideal in $R$, $\Vol^{J_\bullet}_F(\underline{I})$ is denoted by $\Vol^J_F(\underline{I})$, and we call it the $F$-volume of the sequence $\underline{I}$ with respect to $J$. 
\end{definition}

We end this section providing an example that  shows that different generators of an ideal do not necessarily give equal volumes, that is, if we take two ideals $I,J$ such that $I\subseteq \sqrt{J}$, and $I =(f_1,\ldots,f_t)=(g_1,\ldots,g_s)$ with $\underline{f} \not= \underline{g}$, then it is possible to have $\Vol_F^J(\underline{f})\not = \Vol_F^J(\underline{g})$. 

\begin{example}\label{ex}
We take $R=K[[x,y]]$ with $K$ an $F$-finite field of characteristic $p=2$. Let $I=(x,y^2)=(x,y^2+x),\;\mathfrak{m}=(x,y),\;\underline{f}=x,y^2$ and $\underline{g}=x,y^2+x$. 

Let us compute $\Vol_F^\mathfrak{m}(\underline{f})$. We note that for $a,\;b, \;e \in \mathbb{N}$,
\begin{align*}
x^a y^{2b} \not\in \mathfrak{m}^{[p^e]}&\Leftrightarrow a \leq p^e-1,\;2b\leq p^e-1\\
&\Leftrightarrow a \leq p^e-1,\;b\leq \frac{p^e-1}{2} \\
&\Leftrightarrow a \leq p^e-1,\;b\leq \left\lfloor \frac{p^e-1}{2} \right\rfloor=\frac{p^e-2}{2}. 
\end{align*} 
Hence, $\V_{\underline{f}}^\mathfrak{m}(p^e)=[0,p^e-1]\times \left[0,\frac{p^e-2}{2}\right]\cap \mathbb{N}^2$. Thus, $|\V_{\underline{f}}^\mathfrak{m}(p^e)|=\frac{p^{2e}}{2}$. Therefore,
\begin{align*}
\Vol_F^\mathfrak{m}(\underline{f})&=\lim\limits_{e \rightarrow \infty}\frac{|\V_{\underline{f}}^\mathfrak{m}(p^e)|}{p^{2e}}\\
&= \lim\limits_{e \rightarrow \infty}\frac{p^{2e}}{2p^{2e}}\\
&=\frac{1}{2}.
\end{align*}

Let us compute $\Vol_F^\mathfrak{m}(\underline{g})$. We note that $\V_{\underline{f}}^\mathfrak{m}(p^e) \subseteq \V_{\underline{g}}^\mathfrak{m}(p^e)$. We show that 
\begin{align*}
\V_{\underline{g}}^\mathfrak{m}(p^e)=\V_{\underline{f}}^\mathfrak{m}(p^e) \cup\left( \left[0,\frac{p^e-2}{2}\right]\times \left( \frac{p^e-2}{2},p^e-1\right]\cap \mathbb{N}^2 \right).
\end{align*}
Since $a,\; b \leq p^e-1$ if $(a,b) \in \V_{\underline{g}}^\mathfrak{m}(p^e)$, it is enough to show that $(\frac{p^e-2}{2},p^e-1) \in \V_{\underline{g}}^\mathfrak{m}(p^e)$, and that $(\frac{p^e}{2},\frac{p^e}{2}) \not \in \V_{\underline{g}}^\mathfrak{m}(p^e)$. 

We have that 
\begin{align*}
x^{\frac{p^e-2}{2}}(x+y^2)^{p^e-1}&=\sum_{i=0}^{p^e-1}\binom{p^e-1}{i}x^{\frac{p^e-2}{2}+p^e-1-i}y^{2i}. 
\end{align*}
But, $2 \not | \binom{p^e-1}{\frac{p^e}{2}-1}$, then $\binom{p^e-1}{\frac{p^e}{2}-1}x^{p^e-1}y^{p^e-2} \not \in \mathfrak{m}^{[p^e]}$. Therefore, $(\frac{p^e-2}{2},p^e-1) \in \V_{\underline{g}}^\mathfrak{m}(p^e)$.

Moreover, $x^{\frac{p^e}{2}}(x+y^2)^{\frac{p^e}{2}}=x^{\frac{p^e}{2}}y^{p^e}+x^{p^e} \in \mathfrak{m}^{[p^e]}$. Therefore, $(\frac{p^e}{2},\frac{p^e}{2}) \not \in \V_{\underline{g}}^\mathfrak{m}(p^e)$.

It follows that 
\begin{align*}
|\V_{\underline{g}}^\mathfrak{m}(p^e)|&=|\V_{\underline{f}}^\mathfrak{m}(p^e)|+\frac{p^e}{2} \cdot \frac{p^e}{2} \\
&=\frac{p^{2e}}{2}+\frac{p^{2e}}{4}\\
&=\frac{3}{4}p^{2e}.
\end{align*}
Thus, we have
\begin{align*}
\Vol_F^\mathfrak{m}(\underline{g})&=\lim\limits_{e \rightarrow \infty}\frac{|\V_{\underline{g}}^\mathfrak{m}(p^e)|}{p^{2e}}\\
&= \lim\limits_{e \rightarrow \infty}\frac{3p^{2e}}{4p^{2e}}\\
&=\frac{3}{4}. 
\end{align*}      
\end{example}

%%%%%%%%%%%%%%%%%%%%%%%%%%%%%%%%%%%%%%%%%%%%%%%%%%%%%%%%%%%%%%
\section{First properties}
%%%%%%%%%%%%%%%%%%%%%%%%%%%%%%%%%%%%%%%%%%%%%%%%%%%%%%%%%%%%%%
In this section we discuss basic properties of the $F$-volumes. In particular, we focus on those properties
that resemble the properties of  the $F$-thresholds.

\begin{proposition}\label{PropBasicProp}
Let $\underline{I}=I_1,\ldots, I_t\subseteq R$ be a sequence of ideals, and $\mathfrak{a}, \;J\subseteq R$ be two ideals such that $I_1,\ldots, I_t \subseteq \sqrt{J}$. Then, the following statements hold. 
\begin{itemize}
\item[(1)] If $J \subseteq \mathfrak{a}$, then $\Vol^{\mathfrak{a}}_F(\underline{I})\leq\Vol^J_F(\underline{I})$;
\item[(2)] $\Vol^{J^{[p]}}_F(\underline{I}) =p^t\Vol^J_F(\underline{I})$.
\item[(3)] Let $\underline{I}_{t-1}=I_1,\ldots, I_{t-1}$. Then,  $\Vol^J_F(\underline{I}) \leq \Vol^J_F(\underline{I}_{t-1}) c^J (I_t)$.  
\end{itemize}
\end{proposition}
\begin{proof}$ $\\
\begin{itemize}
\item[(1)]  Since $J \subseteq \mathfrak{a}$, we have $\V_{\underline{I}}^\mathfrak{a}(p^e) \subseteq  \V_{\underline{I}}^J(p^e)$ for every $e \in \mathbb{N}$. Thus, $|\V_{\underline{I}}^\mathfrak{a}(p^e)| \leq |  \V_{\underline{I}}^J(p^e)|$. Therefore, $\Vol^\mathfrak{a}_F(\underline{I}) \leq \Vol^J_F(\underline{I})$.

\item[(2)]  We have that $(J^{[p]})^{[p^e]}=J^{[p^{e+1}]}$, then $ \V_{\underline{I}}^{J^{[p]}}(p^e)= \V_{\underline{I}}^J(p^{e+1})$. Hence $\frac{|\V_{\underline{I}}^{J^{[p]}}(p^e)|}{p^{et}}=\frac{p^t|\V_{\underline{I}}^J(p^{e+1})|}{p^{(e+1)t}}$. Therefore, 
$\Vol^{J^{[p]}}_F(\underline{I}) =p^t \Vol^J_F(\underline{I})$.

\item[(3)] Let $a \in \V_{\underline{I}}^J(p^e)$. Then, $\underline{I}^a  \nsubseteq J^{[p^e]}$. We have that $a_t \leq \nu^J_{I_t}(p^e)$; otherwise, $\underline{I}^a  \subseteq J^{[p^e]}$, which is a contradiction.
Then,  $ \V_{\underline{I}}^J(p^e)\subseteq  \V_{\underline{I}_{t-1}}^J(p^e)\times \{0,\ldots, \nu^J_{I_t}(p^e) \}$.  
Then,
\begin{align*}
\Vol^J_F(\underline{I})& =\lim\limits_{e\to\infty}\frac{|\V_{\underline{I}}^J(p^e)|}{p^{et}}	 \\
&\leq  \lim\limits_{e\to\infty}\frac{|\V_{\underline{I}_{t-1}}^J(p^e)\times   \{0,\ldots, \nu^J_{I_t}(p^e) \} |}{p^{et}}	\\
&=\lim\limits_{e\to\infty}\frac{|\V_{\underline{I}_{t-1}}^J(p^e)  |}{p^{e(t-1)}}	 
\lim\limits_{e\to\infty}\frac{ \nu^J_{I_t}(p^e)+1}{p^{e}}	 \\
&= \Vol^J_F(\underline{I}_{t-1}) c^J (I_t).
\end{align*}
\end{itemize}
\end{proof}

We now show that $F$-volumes are not affected by integral closure.

\begin{proposition}
Let $\underline{I}=I_1,\ldots, I_t\subseteq R$ be a sequence of ideals, and $J_\bullet=\{J_{p^e}\}_{e \in \NN}$ be a $p$-family of ideals in $R$ such that $I_1,\ldots, I_t\subseteq \sqrt{J_1}$. Then, $\Vol^{J_\bullet}_F(\underline{I})=\Vol^{J_\bullet}_F(\overline{I_1},I_2,\ldots,I_t)$, where $\overline{I_1}$ denotes the integral closure of $I_1$. As a consequence, $\Vol^{J_\bullet}_F(\underline{I})=\Vol^{J_\bullet}_F(\overline{I_1},\ldots,\overline{I_t})$.
\end{proposition}
\begin{proof}
Since $I_1 \subseteq \overline{I_1}$ is a reduction, there exists $\ell \in \NN_{>0}$ such that $\overline{I_1}^n \subseteq I_1^{n- \ell}$ for every $n\in\NN$. We consider the set 
$$
H=\{\beta \in \NN^{t-1}\;|\; \exists \beta_1 \in \NN \; \hbox{such that} \; (\beta_1,\beta) \in \V_{\overline{I_1},I_2,\ldots,I_t}^{J_\bullet}(p^e) \backslash \V_{ \underline{I}}^{J_\bullet}(p^e) \}.
$$
For all $\beta \in H$, we denote $b_\beta$ as the largest nonnegative integer such that $$(b_\beta,\beta) \in \V_{\overline{I_1},I_2,\ldots,I_t}^{J_\bullet}(p^e) \backslash \V_{ \underline{I}}^{J_\bullet}(p^e).$$

We show that $\V_{\overline{I_1},I_2,\ldots,I_t}^{J_\bullet}(p^e) \backslash \V_{ \underline{I}}^{J_\bullet}(p^e) \subseteq \bigcup_{\beta \in H}(\NN \cap(b_\beta-\ell,b_\beta]) \times \{\beta\}$. Let $(a_1,\ldots,a_t)$ be an element of $\V_{\overline{I_1},I_2,\ldots,I_t}^{J_\bullet}(p^e) \backslash \V_{ \underline{I}}^{J_\bullet}(p^e)$. Thus, $a=(a_2,\ldots, a_t) \in H$. We have that $\overline{I_1}^{b_a}I_2^{a_2}\cdots I_t^{a_t} \subseteq I_1^{b_a-\ell}I_2^{a_2}\cdots I_t^{a_t}$. We deduce that $I_1^{b_a-\ell}I_2^{a_2}\cdots I_t^{a_t} \not \subseteq J_{p^e}$. But, $I_1^{a_1}I_2^{a_2}\cdots I_t^{a_t} \subseteq J_{p^e}$. As a consequence,  $b_a-\ell < a_1 \leq b_a$. Therefore, $(a_1,a) \in \bigcup_{\beta \in H}(\NN \cap (b_\beta-\ell,b_\beta]) \times \{\beta\}$.

We note that $\V_{ \underline{I}}^{J_\bullet}(p^e) \subseteq \V_{\overline{I_1},I_2,\ldots,I_t}^{J_\bullet}(p^e)$. Thus,     
\begin{align*}
|\V_{\overline{I_1},I_2,\ldots,I_t}^{J_\bullet}(p^e)|-|\V_{ \underline{I}}^{J_\bullet}(p^e)|=|\V_{\overline{I_1},I_2,\ldots,I_t}^{J_\bullet}(p^e) \backslash \V_{ \underline{I}}^{J_\bullet}(p^e)|\leq \ell|H| \leq   \ell \prod_{i=2}^t|\V_{I_i}^{J_\bullet}(p^e)|, 
\end{align*}
where the last inequality follows because $H \subseteq \V_{I_2,\ldots,I_t}^{J_\bullet}(p^e)$.
Since 
\begin{align*}
 \lim\limits_{e\to\infty}  \frac{\prod_{i=2}^t|\V_{I_i}^{J_\bullet}(p^e)|}{p^{e(t-1)}}= \prod_{i=2}^t\Vol^{J_\bullet}_F(I_i), 
\end{align*}
we obtain that
\begin{align*}
0\leq \lim\limits_{e\to\infty}\frac{ |\V_{\overline{I_1},I_2,\ldots,I_t}^{J_\bullet}(p^e)|}{p^{et}}-\frac{| \V_{ \underline{I}}^{J_\bullet}(p^e)|}{p^{et}}\leq 
\lim\limits_{e\to\infty} \frac{\ell \prod_{i=2}^t|\V_{I_i}^{J_\bullet}(p^e)|}{p^{et}}=0. 
\end{align*}
Therefore,
\begin{align*}
\Vol^{J_\bullet}_F(\underline{I})=\lim\limits_{e\to\infty}\frac{|\V_{ \underline{I}}^{J_\bullet}(p^e)|}{p^{et}}=\lim\limits_{e\to\infty}\frac{ |\V_{\overline{I_1},I_2,\ldots,I_t}^{J_\bullet}(p^e)|}{p^{et}}=\Vol^{J_\bullet}_F(\overline{I_1},I_2,\ldots,I_t).
\end{align*}
\end{proof}

We now start describing objects that give us an alternative description to $F$-volumes. These descriptions plays an important role in the following sections.

\begin{definition} 
Let $\underline{I}=I_1,\ldots, I_t\subseteq R$ be a sequence of ideals, and $J_\bullet=\{J_{p^e}\}_{e \in \NN}$ be a $p$-family of ideals in $R$ such that $I_1,\ldots, I_t\subseteq \sqrt{J_1}$. We take
\begin{align*}
B^{J_\bullet}(\underline{I};p^e)=\bigcup_{a\in \V_{\underline{I}}^{J_\bullet}(p^e)} [0,a_1/p^e] \times \ldots\times [0,a_t/p^e].
\end{align*}
If $\underline{f}=f_1,\ldots,f_t$ is a sequence of elements of $R$ such that $f_1,\ldots, f_t\in \sqrt{J_1}$, we use $B^{J_\bullet}(\underline{f};p^e)$ to denote $B^{J_\bullet}(\underline{I};p^e)$ where $\underline{I}=f_1R,\ldots, f_tR$. In case that the $p$-family is $J_\bullet=\{J^{[p^e]}\}_{e \in \NN}$ where $J$ is an ideal in $R$, $B^{J_\bullet}(\underline{I};p^e)$ is denoted by $B^J(\underline{I};p^e)$. 
\end{definition}

\begin{remark} \label{remarkanalogous}

Analogous to Definition \ref{def1}, we take
\begin{align*}
\widetilde{\V}_{\underline{I}}^{J_\bullet}(p^e)=\{(a_1,\ldots, a_t)\in\NN_{> 0}^t\; | \; I_1^{a_1}\cdots I_t^{a_t}\not\subseteq J_{p^e}\},
\end{align*}
and
\begin{align*}
\widetilde{B}^{J_\bullet}(\underline{I};p^e)=\bigcup_{a\in \widetilde{\V}_{\underline{I}}^{J_\bullet}(p^e)} [0,a_1/p^e] \times \ldots\times [0,a_n/p^e].
\end{align*}
If the $p$-family is $J_\bullet=\{J^{[p^e]}\}_{e \in \NN}$ with $J$ an ideal of $R$, we denote $\widetilde{\V}_{\underline{I}}^{J_\bullet}(p^e)$ and $\widetilde{B}^{J_\bullet}(\underline{I};p^e)$  by $\widetilde{\V}_{\underline{I}}^J(p^e)$ and $\widetilde{B}^J(\underline{I};p^e)$  respectively.   

We obtain
	\[\Vol(\widetilde{B}^{J_\bullet}(\underline{I};p^e))=\frac{1}{p^{et}}\vert \widetilde{\V}_{\underline{I}}^{J_\bullet}(p^e)\vert,\]
by 	dividing $\widetilde{B}^{J_\bullet}(\underline{I};p^e)$ in $t$-cubes of volume $1/p^{et}$ and counting the number of $t$-cubes.
	
\end{remark}

\begin{proposition}\label{PropOtherLim}
Let $\underline{I}=I_1,\ldots, I_t\subseteq R$ be a sequence of ideals, and $J_\bullet=\{J_{p^e}\}_{e \in \NN}$ be a $p$-family of ideals in $R$ such that $I_1,\ldots, I_t\subseteq \sqrt{J_1}$. Then,
$$
\Vol^{J_\bullet}_F(\underline{I})=\lim\limits_{e\to\infty}\frac{|\widetilde{\V}_{\underline{I}}^{J_\bullet}(p^e)|}{p^{et}}.
$$
\end{proposition}
\begin{proof}
We note that $\widetilde{\V}_{\underline{I}}^{J_\bullet}(p^e)\subseteq \V_{\underline{I}}^{J_\bullet}(p^e)$ and
\begin{align*}
\V_{\underline{I}}^{J_\bullet}(p^e)\setminus \widetilde{\V}_{\underline{I}}^{J_\bullet}(p^e)&=\{(a_1,\ldots, a_t)\in\NN^t\; | \; 
	I_1^{a_1} \cdots I_t^{a_t}\not\subseteq J_{p^e} \; \& \; \exists i \hbox{ such that }a_i=0\}\\
& \subseteq \{(a_1,\ldots, a_t)\in\NN^t\; | \; \forall i, \,
	a_i\leq |\V_{I_i}^{J_\bullet}(p^e)|-1    \; \& \; \exists i \hbox{ such that }a_i=0\}.\\  
\end{align*}
Then,
$$
|\V_{\underline{I}}^{J_\bullet}(p^e)|-|\widetilde{\V}_{\underline{I}}^{J_\bullet}(p^e)|=
|\V_{\underline{I}}^{J_\bullet}(p^e)\setminus \widetilde{\V}_{\underline{I}}^{J_\bullet}(p^e)| \leq  \sum^t_{i=1} \left( \prod_{i\neq j} |\V_{I_j}^{J_\bullet}(p^e)| \right). 
$$ 
We obtain that 
\begin{align*}
0\leq \lim\limits_{e\to\infty}\frac{ |\V_{\underline{I}}^{J_\bullet}(p^e)|}{p^{et}}-\frac{|\widetilde{\V}_{\underline{I}}^{J_\bullet}(p^e)|}{p^{et}} & \leq 
\lim\limits_{e\to\infty}  \frac{\sum^t_{i=1} \left( \prod_{i\neq j} |\V_{I_j}^{J_\bullet}(p^e)| \right)}{p^{et}} \\
&= \sum^t_{i=1} \lim\limits_{e\to\infty}  \frac{  \left( \prod_{i\neq j} |\V_{I_j}^{J_\bullet}(p^e)| \right)}{p^{et}}=0,  
\end{align*}
where the last equality follows from the fact that  
$$
  \lim\limits_{e\to\infty}  \frac{ \prod_{i\neq j} |\V_{I_j}^{J_\bullet}(p^e)|}{p^{e(t-1)}}= \prod_{i\neq j} \Vol^{J_\bullet}_F(I_j).  
$$
We conclude that 
$$
 \lim\limits_{e\to\infty}\frac{ |\V_{\underline{I}}^{J_\bullet}(p^e)|}{p^{et}}=\lim\limits_{e\to\infty}\frac{|\widetilde{\V}_{\underline{I}}^{J_\bullet}(p^e)|}{p^{et}}.
 $$
\end{proof}

We end this section with an upper bound for the $F$-volume in terms of the $F$-threshold.  
\begin{corollary}\label{CorFactorial}
Let $\underline{I}=I_1,\ldots, I_t\subseteq R$ be a sequence of ideals, and $J\subseteq R$ be an ideal such that $I_1,\ldots, I_t \subseteq \sqrt{J}$. Let $c=c^J(I_1+ \cdots +I_t)$. Then, $\Vol^J_F(\underline{I})\leq \frac{c^t}{t!}$. 
\end{corollary}
\begin{proof}
Let $I=I_1+ \cdots +I_t$. For $\alpha, e\in \NN$, we have $I^\alpha\not\subseteq J^{[p^e]}$ if and only if there exists
$a=(a_1,\ldots, a_t)\in \NN^t$ with $a_1+\cdots+a_t=\alpha$ such that $\underline{I}^a\not\subseteq J^{[p^e]}$.
Hence, $\nu_I^J(p^e)=\max\{|a|\mid a\in \V_{\underline{I}}^J(p^e)\}$.
Additionally, for every $a \in \widetilde{B}^J(\underline{I};p^e)$ there exists $b\in V_{\underline{I}}^J(p^e)$ such that $|p^ea|\leq |b|$.
Consequently, we have that $\frac{\nu_I^J(p^e)}{p^e}\geq\max\{|a|\mid a\in \widetilde{B}^J(\underline{I};p^e)\}$.

Let $\nu(p^e)=\frac{\nu_I^J(p^e)}{p^e}$. 
We use $H(p^e)$ to denote the set $\{(x_1,\ldots,x_t)\in \mathbb{R}_{\geq 0}^t\;|\;x_1+\ldots+x_t \leq \nu(p^e)\}$. Then we have that $\widetilde{B}^J(\underline{I};p^e)\subseteq H(p^e)$. 
Thus, $\operatorname{Vol}(\widetilde{B}^J(\underline{I};p^e))\leq \Vol(H(p^e))=\frac{\nu(p^e)^t}{t!}$.     
As a consequence, we have that
$$
\frac{1}{p^{et}}\vert \widetilde{\V}_{\underline{I}}^J(p^e)\vert\leq \frac{\nu(p^e)^t}{t!}
$$ by Remark \ref{remarkanalogous}. Since $ \lim\limits_{e\to\infty}\nu(p^e)=c$, it follows that $\Vol^J_F(\underline{I})\leq \frac{c^t}{t!}$ by Proposition \ref{PropOtherLim}. 
\end{proof}

%%%%%%%%%%%%%%%%%%%%%%%%%%%%%%%%%%%%%%%%%%%%%%%%%%%%%%%%%%%%%%
\section{Properties for $F$-pure rings}\label{SecFpure}
%%%%%%%%%%%%%%%%%%%%%%%%%%%%%%%%%%%%%%%%%%%%%%%%%%%%%%%%%%%%%%
In this section we focus on $F$-pure rings. In particular, in Proposition \ref{PropositionMeasureB} we prove that the $F$-volume is in fact the volume of an object in a real space.  We also provide a few properties that hold only in this case.

\begin{proposition} \label{crece}
Let $\underline{I}=I_1,\ldots, I_t\subseteq R$ be a sequence of ideals, and $J\subseteq R$ be an ideal such that $I_1,\ldots, I_t\subseteq \sqrt{J}$. 
If $R$ is $F$-pure, then 
$$
B^J(\underline{I};p^e)\subseteq B^J(\underline{I};p^{e+1}).
$$
\end{proposition}
\begin{proof}
For every element $a$ in $\V_{\underline{I}}^J(p^e)$, we have that $\underline{I}^a \not \subseteq J^{[p^e]}$. Since $R$ is an $F$-pure ring, $(\underline{I}^a)^{[p]}\not \subseteq J^{[p^{e+1}]}$. As a consequence, $\underline{I}^{pa} \not \subseteq J^{[p^{e+1}]}$, thus $pa \in \V_{\underline{I}}^J(p^{e+1})$.

In addition, we have that
\begin{align*}
\left[0,\frac{a_1}{p^e}\right] \times \cdots \times \left[0,\frac{a_t}{p^e}\right] = \left[0,\frac{pa_1}{p^{e+1}}\right]\times \cdots \times \left[0,\frac{pa_t}{p^{e+1}}\right].
\end{align*}
Therefore,
\begin{align*}
B^J(\underline{I};p^e)\subseteq B^J(\underline{I};p^{e+1}).
\end{align*}       
\end{proof}

\begin{remark}\label{remarkcrece}
Taking the same condition of Proposition \ref{crece}, we have that $$\widetilde{B}^J(\underline{I};p^e)\subseteq \widetilde{B}^J(\underline{I};p^{e+1}).$$ 
\end{remark}

\begin{definition} 
Suppose that $R$ is an $F$-pure ring. Let $\underline{I}=I_1,\ldots, I_t\subseteq R$ be a sequence of ideals, and $J_\bullet=\{J_{p^e}\}_{e \in \NN}$ be a $p$-family of ideals in $R$ such that $I_1,\ldots, I_t\subseteq \sqrt{J_1}$. We take
$$
B^{J_\bullet}(\underline{I})= \bigcup_{e\in\NN} B^{J_\bullet}(\underline{I};p^e).
$$
If $\underline{f}=f_1,\ldots,f_t$ is a sequence of elements of $R$ such that $f_1,\ldots, f_t\in \sqrt{J_1}$, we use $B^{J_\bullet}(\underline{f})$ to denote $B^{J_\bullet}(\underline{I})$ where $\underline{I}=f_1R,\ldots, f_tR$. In case that the $p$-family is $J_\bullet=\{J^{[p^e]}\}_{e \in \NN}$ where $J$ is an ideal in $R$, $B^{J_\bullet}(\underline{I})$ is denoted by $B^J(\underline{I})$. 
\end{definition}

Suppose that $(R,\m,K)$ is an $F$-finite regular local ring.
Let $\underline{I}=I_1,\ldots, I_t\subseteq R$ be a sequence of ideals. 
The mixed test  ideals  $\tau(I^{a_1}_1 \cdots  I^{a_t}_t)$
are important objects studied in birational geometry \cite{HY2003,BMS-MMJ,FelipeConstReg}.
The set $B^\m(\underline{I})$ is the first constancy region for these ideals \cite{FelipeConstReg}.
Furthermore, $B^J(\underline{I})$ is the union of the constancy regions whose test ideal is not contained in $J$.

If the ring is not regular, then $B^\m(\underline{I})$ is no longer a constancy region. To see this, it suffices to look at the case where $t=1$, and take any example where the $F$-threshold $c^{\m}(\m)\neq \hbox{fpt}(\m)$  \cite{MOY}.

\begin{remark}\label{RemC}
Suppose that $R$ is an $F$-pure ring. Let $\underline{I}=I_1,\ldots, I_t\subseteq R$ be a sequence of  ideals, and $J \subseteq R$ be a proper ideal such that $I_1,\ldots, I_t\subseteq\sqrt{J}$. 
We  have that 
\begin{align*}
B^J(\underline{I})=\left(\bigcup_{e\in\NN} \widetilde{B}^J(\underline{I};p^e) \right)\bigcup C,
\end{align*}
where $C=\bigcup_{e\in \NN} \bigcup_{a\in \V_{\underline{I}}^{J_\bullet}(p^e)\setminus \widetilde{\V}_{\underline{I}}^{J_\bullet}(p^e)}  [0,a_1/p^e] \times \ldots\times [0,a_t/p^e].$ We note that $C$ has measure zero as subset of $\mathbb{R}^t$.
\end{remark}

The following result justifies in part the name of $F$-volume.

\begin{proposition}\label{PropositionMeasureB}
Suppose that $R$ is an $F$-pure ring.
Let $\underline{I}=I_1,\ldots,I_t\subseteq R$ be a sequence of ideals, and $J\subseteq R$ be an ideal such that $I_1,\ldots, I_t\subseteq  \sqrt{J}$. 
Then, $B^J(\underline{I})$ is a measurable set. Furthermore,
$$
\Vol(B^J(\underline{I}))=\lim\limits_{e\to\infty} \frac{1}{p^{et}} |\widetilde{\V}_{\underline{I}}^J(p^e)|.
$$
In particular, 
$$
\Vol^J_F(\underline{I})=\Vol(B^J(\underline{I})).
$$
\end{proposition}
\begin{proof}
Since  $ \widetilde{B}^J(\underline{I};p^e)$ is measurable for every $e\in \NN$, we  conclude that $\widetilde{B}^J(\underline{I})=\bigcup_{e\in\NN} \widetilde{B}^J(\underline{I};p^e)$ is also measurable.
Similarly, $B^J(\underline{I})$ is measurable.

From Remark \ref{remarkanalogous}, we recall that
$$
\Vol(\widetilde{B}^J(\underline{I};p^e)) =\frac{1}{p^{et}}\vert \widetilde{\V}_{\underline{I}}^J(p^e)\vert.
$$	
	Since $\widetilde{B}^J(\underline{I};p^e)\subseteq \widetilde{B}^J(\underline{I};p^{e+1})$ for every $e\in \NN$, we conclude that 
	$$
\Vol(\widetilde{B}^J(\underline{I}))=\lim\limits_{e\to\infty} \frac{1}{p^{et}} |\widetilde{\V}_{\underline{I}}^J(p^e)|=\Vol^J_F(\underline{I}),
$$
where the last equality follows from Proposition  \ref{PropOtherLim}.
Hence,
$$
\Vol(B^J(\underline{I}))=\Vol(\widetilde{B}^J(\underline{I})\cup C)
=\Vol(\widetilde{B}^J(\underline{I}))=\Vol^J_F(\underline{I})
$$
by Remark \ref{RemC}.
\end{proof}

\begin{remark} \label{remarkincreasing}
Suppose that $R$ is an $F$-pure ring. Let $\underline{I}=I_1,\ldots, I_t\subseteq R$ be a sequence of ideals, and let $J$ be an ideal in $R$ such that $I_1,\ldots, I_t\subseteq \sqrt{J}$. From Remarks \ref{remarkanalogous} and \ref{remarkcrece}, we have that the sequence $\left\{ \frac{|\widetilde{\V}_{\underline{I}}^J(p^e)|}{p^{et}} \right\}_{e\  \in \NN}$ is increasing. 
\end{remark}

\begin{lemma}\label{remarkbound}
	Let $\underline{I}=I_1,\ldots, I_t\subseteq R$ be a sequence of ideals, and $J_\bullet=\{J_{p^e}\}_{e \in \NN}$ be a $p$-family of ideals in $R$ such that $I_1,\ldots, I_t\subseteq \sqrt{J_1}$. Then, there exists $u \in \RR_{\geq 0}$ such that $$\frac{|\widetilde{\V}_{\underline{I}}^{J_{p^e}}(p^{e_1+e_2})|}{p^{(e_1+e_2)t}}\leq\frac{|\widetilde{\V}_{\underline{I}}^{J_{p^e}}(p^{e_1})|}{p^{e_1t}}+\frac{p^{e(t-1)}u}{p^{e_1}}$$ for every $e,e_1,e_2 \in \NN$.
\end{lemma} 
\begin{proof}
	For each $e\in \NN$ and $n\in\{1,...,t\}$, let $\ell_{e,n}=\min\{\ell\mid I_n^{\ell}\subseteq J_{p^e}\}$ and consider the sets
	\begin{itemize}
	 \item$\mathcal{B}(\underline{I})_{e_1}^e=\frac{1}{p^{e_1}}\NN^t\cap\left(\bigcup_{j=1}^t\left(\prod_{i=1}^{j-1}[0,\mu(I_i)\ell_{e,i}]\times \{0\}\times\prod_{i=j+1}^t[0,\mu(I_i)\ell_{e,i}]\right)\right)$
	\item $\mathcal{L}_{e_1,e_2}^e=H_{e_1, e_2}\left(  \bigcup^\mu_{j=0} \left(\partial \left(\frac{1}{p^{e_1}}\V_{\underline{I}}^{J_{p^e}}(p^{e_1})\cup \mathcal{B}(\underline{I})_{e_1}^e\right)+\frac{j}{p^{e_1}}\mathbf{1}\right)\right)$
	\end{itemize}
	where $\mu=\max\{\mu(I_1),\ldots,\mu(I_t)\}$.
	
	We claim that
 $$
	\frac{1}{p^{e_1+e_2}}\widetilde{\V}_{\underline{I}}^{J_{p^e}}(p^{e_1+e_2})\subseteq H_{e_1,e_2}\left(\frac{1}{p^{e_1}}\widetilde{\V}_{\underline{I}}^{J_{p^e}}(p^{e_1})\right)\cup\mathcal{L}_{e_1,e_2}^e.$$
Let $x\in\frac{1}{p^{e_1+e_2}}\widetilde{\V}_{\underline{I}}^{J_\bullet}(p^{e_1+e_2})$. Suppose that $x\not\in \mathcal{L}_{e_1,e_2}^e.$ From Lemma \ref{lem2}, we have that
$$
	\frac{1}{p^{e_1+e_2}}\V_{\underline{I}}^{J_{p^e}}(p^{e_1+e_2})\subseteq H_{e_1,e_2}\left(\frac{1}{p^{e_1}}\V_{\underline{I}}^{J_{p^e}}(p^{e_1})\right)\cup\mathcal{L}_{e_1,e_2}^e.$$
	We note that $\widetilde{\V}_{\underline{I}}^{J_{p^e}}(p^{e_1+e_2}) \subseteq \V_{\underline{I}}^{J_{p^e}}(p^{e_1+e_2})$. Since $x\not\in \mathcal{L}_{e_1,e_2}^e$, $x \in H_{e_1,e_2}\left(\frac{1}{p^{e_1}}\V_{\underline{I}}^{J_{p^e}}(p^{e_1})\right)$. Then, there exists $y\in \frac{1}{p^{e_1}}\V_{\underline{I}}^{J_{p^e}}(p^{e_1})$ such that 
	$y_i-\frac{1}{p^{e_1}}<x_i\leq y_i$
	for every $i$. 
	Since $x_i>0$ for every $i$, $y_i>0$ for every $i$. 
	Hence $y\in \frac{1}{p^{e_1}}\widetilde{\V}_{\underline{I}}^{J_{p^e}}(p^{e_1})$  and
	$x\in H_{e_1,e_2}\left(\frac{1}{p^{e_1}}\widetilde{\V}_{\underline{I}}^{J_{p^e}}(p^{e_1})\right)$.
	Therefore,
	\[
	\frac{1}{p^{e_1+e_2}}\widetilde{\V}_{\underline{I}}^{J_{p^e}}(p^{e_1+e_2})\subseteq H_{e_1,e_2}\left(\frac{1}{p^{e_1}}\widetilde{\V}_{\underline{I}}^{J_{p^e}}(p^{e_1})\right)\cup\mathcal{L}_{e_1,e_2}^e.\]
	Consequently, we have that 
	$$\frac{|\widetilde{\V}_{\underline{I}}^{J_{p^e}}(p^{e_1+e_2})|}{p^{(e_1+e_2)t}}\leq\frac{|\widetilde{\V}_{\underline{I}}^{J_{p^e}}(p^{e_1})|}{p^{e_1t}}+\frac{(\mu+1)\sum_{n=1}^t\left(\prod_{j=1}^{n-1}(\mu(I_j)\ell_{e,j}+1)\prod_{j=n+1}^t(\mu(I_j)\ell_{e, j}+1)\right)}{p^{e_1}}.$$
	
	On the other hand, since $I_n^{\mu(I_n)\ell_{0,n}p^e}\subseteq J_1^{[p^e]} \subseteq J_{p^e}$, we have that $\ell_{e,n}\leq\mu(I_n)\ell_{0,n}p^e$. Thus, if 
	$u=(\mu+1)\sum_{n=1}^t\left(\prod_{j=1}^{n-1}(\mu(I_j)^2\ell_{0,j}+1)\prod_{j=n+1}^t(\mu(I_j)^2\ell_{0,j}+1)\right)$, we obtain
	$$\frac{|\widetilde{\V}_{\underline{I}}^{J_{p^e}}(p^{e_1+e_2})|}{p^{(e_1+e_2)t}}\leq\frac{|\widetilde{\V}_{\underline{I}}^{J_{p^e}}(p^{e_1})|}{p^{e_1t}}+\frac{p^{e(t-1)}u}{p^{e_1}}.$$
	
\end{proof}

We now introduce another basic property for $F$-volumes for $F$-pure rings.

\begin{proposition}
Suppose that $R$ is an $F$-pure ring. Let $\underline{I}=I_1,\ldots, I_t\subseteq R$ be a sequence of ideals, and $J_\bullet=\{J_{p^e}\}_{e \in \NN}$ be a $p$-family of ideals in $R$ such that $I_1,\ldots, I_t\subseteq \sqrt{J_1}$. Then, 
$$\Vol^{J_\bullet}_F(\underline{I})=\lim\limits_{e\to\infty} \frac{\Vol^{J_{p^e}}_F(\underline{I})}{p^{et}}.$$
\end{proposition}

\begin{proof}

For every $e \in \NN$ we have $J_{p^e}^{[p]} \subseteq J_{p^{e+1}}$. Thus, $\Vol^{J_{p^{e+1}}}_F(\underline{I}) \leq \Vol^{J_{p^e}^{[p]}}_F(\underline{I})=p^t \cdot \Vol^{J_{p^e}}_F(\underline{I})$. Hence, $$0 \leq \frac{\Vol^{J_{p^{e+1}}}_F(\underline{I})}{p^{(e+1)t}} \leq \frac{\Vol^{J_{p^e}}_F(\underline{I})}{p^{et}},$$ 
which shows the sequence $\left\{ \frac{\Vol^{J_{p^e}}_F(\underline{I})}{p^{et}} \right\}_{e\  \in \NN}$ is decreasing, and bounded below by zero. As a consequence, it converges to a limit as $e$ approaches infinity.

Note that, for every nonnegative integer $e$, we have that 
\begin{align*}
\widetilde{\V}_{\underline{I}}^{J_\bullet}(p^e)&=\{(a_1,\ldots,a_t)\in \NN_{>0}^t\;|\; I_1^{a_1}\cdots I_t^{a_t} \not \subseteq J_{p^e}\}\\
&=\{(a_1,\ldots,a_t)\in \NN_{>0}^t\;|\; I_1^{a_1}\cdots I_t^{a_t} \not \subseteq J_{p^e}^{[p^0]}\}\\
&=\widetilde{\V}_{\underline{I}}^{J_{p^e}}(p^0).
\end{align*}
By Lemma \ref{remarkbound}, there exists $u \in \RR_{\geq0}$ (that does not depend on $e$) such that for every nonnegative integer $s$, we have
\begin{align*}
\frac{|\widetilde{\V}_{\underline{I}}^{J_{p^e}}(p^s)|}{p^{st}}-\frac{|\widetilde{\V}_{\underline{I}}^{J_{p^e}}(p^0)|}{p^{0t}} \leq \frac{p^{e(t-1)}u}{p^{0}}.
\end{align*}

Since $R$ is a $F$-pure ring, the sequence $\left\{ \frac{|\widetilde{\V}_{\underline{I}}^J(p^s)|}{p^{st}} \right\}_{s \geq 0}$ is increasing by Remark \ref{remarkincreasing}. As a consequence,    
\begin{align*}
0 \leq \frac{|\widetilde{\V}_{\underline{I}}^{J_{p^e}}(p^s)|}{p^{st}}-|\widetilde{\V}_{\underline{I}}^{J_{p^e}}(p^0)|\leq p^{e(t-1)}u.
\end{align*}
Thus,
\begin{align*}
0 \leq \frac{|\widetilde{\V}_{\underline{I}}^{J_{p^e}}(p^s)|}{p^{st}}-|\widetilde{\V}_{\underline{I}}^{J_\bullet}(p^e)|\leq p^{e(t-1)}u.
\end{align*}
We take limit over $s$ to get
\begin{align*}
0 \leq \Vol^{J_{p^e}}_F(\underline{I})-|\widetilde{\V}_{\underline{I}}^{J_\bullet}(p^e)| \leq p^{e(t-1)}u,
\end{align*}
dividing by $p^{et}$ gives
\begin{align*}
0 \leq \frac{\Vol^{J_{p^e}}_F(\underline{I})}{p^{et}}-\frac{|\widetilde{\V}_{\underline{I}}^{J_\bullet}(p^e)|}{p^{et}} \leq \frac{u}{p^{e}}.
\end{align*}
Taking limit over $e$ we conclude that
\begin{align*}
 \lim\limits_{e\to\infty} \frac{\Vol^{J_{p^e}}_F(\underline{I})}{p^{et}} =\Vol^{J_\bullet}_F(\underline{I}).
\end{align*}

\end{proof}

\begin{proposition}
Suppose that $R$ is an $F$-finite regular ring.
Let $\underline{I}=I_1,\ldots, I_t$ be a sequence of ideals in $R$, $J$ be an ideal of $R$, and $\{J_i\}_i$ be a family of ideals such that $J=\bigcap_iJ_i$ and $I_1,\ldots, I_t \subseteq\sqrt{J}$. Then, $B^J(\underline{I})=\bigcup_iB^{J_i}(\underline{I})$. In particular, $\Vol(B^J(\underline{I}))=\Vol(\bigcup_iB^{J_i}(\underline{I}))$.
\end{proposition}
\begin{proof}
We show that $\bigcup_i\V_{\underline{I}}^{J_i}(p^e) =\V_{\underline{I}}^J(p^e)$ for every nonnegative integer $e$. We claim that  $\bigcup_i\V_{\underline{I}}^{J_i}(p^e) \subseteq \V_{\underline{I}}^J(p^e)$. Indeed, let $a \in \V_{\underline{I}}^{J_i}(p^e)$ for some $i$, then $\underline{I}^a \not \subseteq J_i^{[p^e]}$. Since $J \subseteq J_i$, we have that $\underline{I}^a \not \subseteq J^{[p^e]}$. Thus, $a \in \V_{\underline{I}}^J(p^e)$.
 
We now prove the other inclusion. Let $a \in \V_{\underline{I}}^J(p^e)$, then we have that $\underline{I}^a \not \subseteq J^{[p^e]}=\left( \bigcap_iJ_i \right) ^{[p^e]}=\bigcap_i J_i^{[p^e]}$. Consequently, there exists $i$ such that $\underline{I}^a \not \subseteq J_i^{[p^e]}$. Thus, $a \in \V_{\underline{I}}^{J_i}(p^e)$. 

It follows that $B^J(\underline{I};p^e)=\bigcup_iB^{J_i}(\underline{I};p^e)$, thus $B^J(\underline{I})=\bigcup_iB^{J_i}(\underline{I})$. Therefore, $\Vol(B^J(\underline{I}))=\Vol(\bigcup_iB^{J_i}(\underline{I}))$.
\end{proof}

\begin{remark} \label{remark3}
Suppose that $R$ is an $F$-pure ring. Let $\underline{I}=I_1,\ldots, I_t\subseteq R$ be a sequence, and $J\subseteq R$ be an ideal such that $I_1,\ldots, I_t \subseteq \sqrt{J}$. If we take $\alpha \in B^J(\underline{I};p^e)$, there exists $\beta \in \V_{\underline{I}}^J(p^e)$ such that each $\alpha_i \leq \frac{\beta_i}{p^e}$. Thus, $\lfloor p^e \alpha_i\rfloor \leq \beta_i$. Since $\underline{I}^\beta \not \subseteq J^{[p^e]}$, $\underline{I}^{\lfloor p^e \alpha\rfloor} \not \subseteq J^{[p^e]}$. Therefore, $\lfloor p^e \alpha\rfloor \in \V_{\underline{I}}^J(p^e)$. 
\end{remark}

%\begin{proposition}
%Suppose that $R$ is an $F$-pure ring.
%Let $\underline{I}=I_1,\ldots, I_t\subseteq R$ be a sequence, and $J\subseteq R$ be an ideal such that $I_1,\ldots, I_t \in \sqrt{J}$. Let $\textbf{1}=(1,\ldots,1)\in \NN_{>0}^t$. Then, $c^J(\underline{I}^\textbf{1})^t \leq \Vol^J_F(\underline{I})$.
%\end{proposition}
%\begin{proof}
%Let $e\in\NN$, and  $\ell=\nu^J_{\underline{I}^\textbf{1}}(p^e)$. Then, $\underline{I}^{\ell \textbf{1}} \not \in J^{[p^e]}$. Thus, $\ell \textbf{1} \in \V_{\underline{I}}^J(p^e)$, and so 
%\begin{align*}
%\left[0,\frac{\ell}{p^e}\right]^t \subseteq B^J(\underline{I};p^e). 
%\end{align*}
%Hence, $\left(\frac{\nu^J_{\underline{I}^\textbf{1}}(p^e)}{p^e}\right)^t \leq \Vol(B^J(\underline{I};p^e))$. Taking limit over $e$ we conclude that $c^J(\underline{I}^\textbf{1})^t \leq \Vol(B^J(\underline{I}))$.   
%\end{proof}

Suppose that $R$ is an $F$-finite regular  ring.
Let $\underline{I}=I_1,\ldots, I_t\subseteq R$ be a sequence of ideals, 
and $I=I_1+\cdots + I_t$.
The mixed test  ideals satisfy the following equation
$$
\tau(I^\lambda)=\sum_{\alpha_1+\cdots + \alpha_t=\lambda} \tau(I^{\alpha_1}_1 \cdots  I^{\alpha_t}_t).
$$
Motivated by this result, we obtain the following similar properties for $F$-thresholds.
This plays an important role to characterize $F$-pure complete intersections in terms of $F$-volume.

\begin{proposition}\label{propsup}
Suppose that $R$ is an $F$-pure ring.
Let $\underline{I}=I_1,\ldots, I_t\subseteq R$ be a sequence, and $J\subseteq R$ be an ideal such that $I_1,\ldots, I_t \subseteq \sqrt{J}$.
Then, 
$$
c^J(I_1+\cdots + I_t)= \sup \{ |\theta| \; | \; \theta\in B^J(\underline{I}) \}.
$$
\end{proposition}
\begin{proof}
Let $\lambda= \sup \{ |\theta| \; | \; \theta \in B^J(\underline{I}) \}$
and $I=I_1+\cdots + I_t$.

Since $I^{\nu^J_I(p^e)}\not\subseteq J^{[p^e]}$, there exists $\alpha=(\alpha_1,\ldots,\alpha_t)\in \NN^t$ such that
$I^{\alpha_1}_1\cdots I^{\alpha_t}_t \not	\subseteq J^{[p^e]}$ and $|\alpha|=\nu^J_I(p^e)$.
Then, $\frac{1}{p^e}\alpha\in B^J(\underline{I})$. We conclude that $\frac{\nu^J_I(p^e)}{p^e} \leq \lambda$ for every $e.$
Then, $c^J(I)\leq 	\lambda.$

We now show the other inequality.
Let $\alpha=(\alpha_1,\ldots,\alpha_t) \in B^J(\underline{I})$. Then, $\alpha=(\alpha_1,\ldots,\alpha_t)\in  B^J(\underline{I};p^e)$ for $e\gg 0$. Then,  $(\lfloor p^e \alpha_1\rfloor,\ldots,\lfloor p^e \alpha_t\rfloor )\in \V_{\underline{I}}^J(p^e)$ for $e\gg 0$ by Remark \ref{remark3}.
We conclude that $I^{\lfloor p^e \alpha_1\rfloor}_1\cdots I^{\lfloor p^e \alpha_t\rfloor}_t\not\subseteq J^{[p^e]}$
Then, $I^{ \lfloor p^e \alpha_1\rfloor+\cdots+\lfloor p^e \alpha_t\rfloor }\not\subseteq J^{[p^e]}$.
Thus, $\lfloor p^e \alpha_1\rfloor+\cdots+\lfloor p^e \alpha_t\rfloor \leq   \nu^J_I(p^e)$ for $e\gg 0$.
We have that
$$
|\alpha|=\lim\limits_{e\to\infty }\frac{\lfloor p^e \alpha_1\rfloor+\cdots+\lfloor p^e \alpha_t\rfloor}{p^e} \leq  
\lim\limits_{e\to\infty }\frac{ \nu^J_I(p^e)}{p^e}=c^J(I).
$$ 
%Since this happens for every $\theta<\lambda$, 
We conclude that $\lambda\leq c^J(I).$
\end{proof}

%\begin{proposition}
%Let $I \subseteq \mathfrak{m}$ be an ideal in $S$, $\underline{f}=f_1,\ldots, f_t$ be minimal generators of $I$, and $s=c^{\m}(I)$. Then, $\Vol(B^{\mathfrak{m}}(\underline{f})) \leq 1-\left(\frac{t-s}{t}\right)^t$.
%\end{proposition}
%\begin{proof}
%We show that $\left[\frac{s}{t},1\right]^t \cap B^{\mathfrak{m}}(\underline{f})= \emptyset$. If $a \in \left[\frac{s}{t},1\right]^t$, then $a_i \geq \frac{s}{t}$ for every $i$, and so we have that $|a|=\sum_{i=1}^{t}a_i \geq s$. Hence, $\tau(I^{|a|})\subseteq \mathfrak{m}$, thus $\tau(\underline{f}^a) \subseteq \mathfrak{m}$. In consequence, $a \not \in  B^{\mathfrak{m}}(\underline{f})$.

%We denote $H$ by the set $[0,1]^t-\left[\frac{s}{t},1\right]^t$. We have that $B^{\mathfrak{m}}(\underline{f}) \subseteq H$, therefore $\Vol(B^{\mathfrak{m}}(\underline{f}))\leq \Vol(H)=1-\left(\frac{t-s}{t}\right)^t$.
%\end{proof}

The following result allows us to obtain the $F$-threshold under special circumstances.

\begin{proposition} \label{srftp}
Suppose that $R$ is an $F$-pure ring.
Let $I,J\subseteq R$ be two ideals such that $I \subseteq J$. 
Let $\underline{f}=f_1,\ldots, f_t$ be minimal generators of $I$. If $\Vol^J_F(\underline{f})=1$, then $c^J(I)=t$. 
\end{proposition}
\begin{proof}
We show that $B^J(\underline{f})=[0,1)^t$. It is enough to prove that $[0,1)^t \subseteq    B^J(\underline{f})$. We proceed by contradiction. We suppose that there exists $a \in [0,1)^t$ such that $a \not \in B^J(\underline{f})$. Thus, $H \cap B^J(\underline{f})=\emptyset$, where $H$ denotes the set $[a_1,1]\times \cdots \times [a_t,1]$. Hence, $B^J(\underline{f}) \subseteq [0,1]^t-H$. It follows that $\Vol^J_F(\underline{f})<1$, and we get a contradiction.

In addition, from Proposition \ref{propsup}, we have that 
\begin{align*}
c^J(I)&= \sup \{ |\theta| \; | \; \theta\in B^J(\underline{f}) \} \\ &=\sup \{ |\theta| \; | \; \theta\in [0,1)^t \}\\ &=t.
\end{align*}

\end{proof}

We now characterize $F$-pure complete intersections in terms of $F$-volumes. 
This is along the same lines of how the $F$-pure threshold of a hypersurface characterizes, via Fedder's Criterion \cite{Fedder'scriterion}, when this variety if $F$-pure.

\begin{theorem}\label{ThmFpureCI}
Suppose that $(R,\m,K)$ is a local regular ring.
Let $I \subseteq \mathfrak{m}$ be an ideal in $R$, and $\underline{f}=f_1,\ldots, f_t$ be minimal generators of $I$. Then, $\Vol^{\mathfrak{m}}_F (\underline{f})=1$ if and only if $I$ is an $F$-pure complete intersection. 
\end{theorem}
\begin{proof}
We suppose that $\Vol^{\mathfrak{m}}_F (\underline{f})=1$. Then $c^{\mathfrak{m}}(I)=t=\mu(I)$ by Proposition \ref{srftp}. Thus, $\mu(I)=c^{\m}(I)\leq\mathrm{ht}(I)\leq\mu(I)$. We conclude that $\mathrm{ht}(I)=\mu(I)$. Hence, $f_1,\ldots, f_t$ is a regular sequence in $R$. Therefore, $I$ is an $F$-pure complete intersection.

For the other direction, we show that $[0,1)^t = B^\mathfrak{m}(\underline{f})$. Let $a \in [0,1)^t$. Then, $\max\{a_i\} \leq \frac{p^e-1}{p^e}$ for some $e\in \mathbb{N}$. Since $I$ is an $F$-pure ideal, $R/I$ is an $F$-pure ring. Since  $f_1,\ldots, f_t$ is a regular sequence in $R$, we have that $\underline{f}^{p^e-1}\not \in \mathfrak{m}^{[p^e]}$ by Fedder's criterion \cite[Proposition $2.1$]{Fedder'scriterion}. Then, $(p^e-1,\ldots,p^e-1)\in \V_{\underline{f}}^J(p^e)$. We conclude that  $a\in B^{\mathfrak{m}}(\underline{f};p^e) \subseteq B^{\mathfrak{m}}(\underline{f}) $.
Therefore,  $\Vol^{\mathfrak{m}}_F (\underline{f})=1$ by Proposition \ref{PropositionMeasureB}.    
\end{proof}

%%%%%%%%%%%%%%%%%%%%%%%%%%%%%%%%%%%%%%%%%%%%%%%%%%%%%%%%%%%%%%
\section{Relations with Hilbert-Kunz multiplicities}\label{SecHKVol}
%%%%%%%%%%%%%%%%%%%%%%%%%%%%%%%%%%%%%%%%%%%%%%%%%%%%%%%%%%%%%%

In this section we relate the $F$-volume with Hilbert-Kunz multiplicities.
This is related to previous work done for $F$-thresholds and these multiplicities \cite{NBS}.
We start proving Theorem \ref{MainHK}.

\begin{theorem}\label{ThmHK}
Suppose that $(R,\m,K)$ is a local ring. Let $\underline{f}=f_1,\ldots , f_t$ be part of a system of parameters for $R$, $I=(\underline{f})$, and $\overline{R}=R/I$. Then, 
\[
\e_{HK}(J;R)   \leq   \e_{HK}(J\overline{R}; \overline{R}) \Vol_F^J(\underline{f})
\]
for any $\m$-primary ideal $J$, such that $I\subseteq J$.
\end{theorem}
\begin{proof}
Let $I=(\underline{f})$ and 
 $\cI_e=(f^{a_1}_1 \cdots f^{a_t}_t\;|\;a\not\in \V_{\underline{f}}^J(p^e) )R$. 
 Then,  $R/\cI_e$ has a filtration $0=N_0\subseteq N_1\subseteq \cdots \subseteq N_m=R/\cI_e$ where $N_{t+1}/N_t$ is a homomorphic image of $R/I$ and $m=| \V_{\underline{f}}^J(p^e)|$.
 Since $J^{[p^e]}$ is $\m$-primary, we have that
 $$
 \lambda(N_{t+1}\otimes_R R/J^{[p^e]})\leq \lambda(N_{t}\otimes_R R/J^{[p^e]})+\lambda(N_{t+1}/N_t\otimes_R R/J^{[p^e]}).
 $$ 
 As a consequence, 
$$
 \lambda(R/\cI_e\otimes_R R/J^{[p^e]})\leq | \V_{\underline{f}}^J(p^e)| \lambda(R/I\otimes_R R/J^{[p^e]}).
$$
 By the definition of $\cI_e$, we have that $\cI_e\subseteq J^{[p^e]}$.
 Then, 
 \begin{align*}
 \lambda (R/J^{[p^e]})&=\lambda (R/\cI_e +J^{[p^e]})\\
 &= \lambda(R/\cI_e\otimes_R R/J^{[p^e]})\\
 &\leq | \V_{\underline{f}}^J(p^e)| \lambda(R/I\otimes_R R/J^{[p^e]})\\
  &\leq | \V_{\underline{f}}^J(p^e)| \lambda(R/(I+J^{[p^e]})).
 \end{align*}
 After dividing by $p^{ed}$, where $d=\dim(R)$, we obtain that 
 $$
  \frac{\lambda (R/J^{[p^e]})}{p^{ed}} \leq \frac{ | \V_{\underline{f}}^J(p^e)| \lambda(R/(I+J^{[p^e]}))}{p^{ed}}
  =\frac{ | \V_{\underline{f}}^J(p^e)|}{p^{et}}\cdot
  \frac{\lambda(R/(I+J^{[p^e]}))}{p^{e(d-t)}}
  =\frac{ | \V_{\underline{f}}^J(p^e)|}{p^{et}} \cdot \frac{\lambda(\overline{R}/J^{[p^e]})}{p^{e(d-t)}}.
 $$
 After taking the limit as $e\to \infty$, we obtain the desired inequality.
\end{proof}

We now recall a conjecture that relates the Hilbert-Kunz multiplicity and $F$-thresholds.

\begin{conjecture}[{\cite{NBS}}]\label{ConjNBS}
Let $(R,\m,K)$ be a local ring. Let $I\subseteq R$ be an ideal  
generated by a part of a system of parameters $(f_1,\ldots, f_\ell).$ Let $\overline{R}=R/I.$
Let $J$ be an $m$-primary ideal.
Then,
$$
\ehk(J)\leq \ehk(J\overline{R}) \frac{(c^J(I))^{\ell}}{\ell^\ell}.
$$
\end{conjecture}

\begin{remark}\label{RemComp}
We recall two inequalities  \cite{NBS} related to Conjecture \ref{ConjNBS} that were previously obtained:
\begin{equation}\label{EqNBS1}
\ehk(J)\leq \ehk(J\overline{R}) \frac{(c^J(I))^{\ell}}{\ell !}.
\end{equation}
and 
\begin{equation}\label{EqNBS2}
\ehk(J)\leq \ehk(J\overline{R}) c^J(f_1) \cdots c^J(f_\ell). 
\end{equation}
By Corollary \ref{CorFactorial}, we have that
$$
 \Vol^J_F(\underline{f})\leq \frac{(c^J(I))^{\ell}}{\ell !}
$$
By Proposition \ref{PropBasicProp}(3), we have that
$$\Vol^J_F(\underline{f}) \leq  c^J(f_1) \cdots c^J(f_\ell).$$ 
Therefore, Theorem \ref{ThmHK} is a refinement of Inequalities \ref{EqNBS1} and \ref{EqNBS2}.
\end{remark}

\section*{Acknowledgments} 
We thank Felipe P\'erez for inspiring conversations. We also  thank Alessandro De Stefani, Daniel J. Hern\'andez, and  Jack Jeffries for helpful comments and suggestions.
We thank the referee for useful comments.
Part of this work was done while the third-named author was  attending a summer undergraduate mathematics research experience organized by the Mexican Academy of Sciences.
%The first author was partially supported by CONACYT Grants 295631 and 284598. The second author was partially supported by CONACYT Grant 284598 and C\'atedras Marcos Moshinsky. The third author was partially supported by the Mexican Academy of Sciences (AMC) and  CONACYT Grant 284598.
%\bibliographystyle{alpha}
%\bibliography{References}

\end{document}